\documentclass[amstex,12pt,reqno]{amsart}
\usepackage{amsmath,amsthm,amscd,amsfonts,amssymb,enumitem}
\usepackage{latexsym}
\usepackage{algorithm}
\usepackage{lipsum} %
\usepackage{algpseudocode}
\usepackage{booktabs}
\usepackage{array}
\usepackage[table]{xcolor}  
\definecolor{headerblue}{RGB}{100,149,237} 

\usepackage{subcaption}

\usepackage{multirow}
\usepackage{xcolor}
\usepackage{lipsum} 
\usepackage{graphicx}
\usepackage{subcaption}
\usepackage{tcolorbox}
\usepackage{booktabs}

\usepackage{stix2}

\usepackage[a4paper, margin=1in]{geometry}  
\usepackage{anysize} 
\usepackage{amsmath} 
\usepackage{array}   
\usepackage{booktabs} 

\usepackage{caption}
\usepackage[table]{xcolor}  

\usepackage[colorlinks=true, linkcolor=blue, anchorcolor=blue, citecolor=blue, filecolor=blue, urlcolor=blue]{hyperref}

\usepackage{gettitlestring}[2009/12/18]
\usepackage{kvoptions}
\newtheorem{thm}{Theorem}[section]
\newtheorem{lemma}[thm]{Lemma}

\theoremstyle{definition}

\newtheorem{remark}{Remark}
\newtheorem{note}[thm]{Note}
\newcommand{\qedwhite}{\hfill \ensuremath{\Box}}
\renewenvironment{proof}{{\raggedright \bfseries Proof.}}{\qedwhite}

\numberwithin{equation}{section}

\renewcommand{\arraystretch}{1.5}
\def\CC{\mathbb C}
\def\NN{\mathbb N}

\def\RR{\mathbb R}

\begin{document}
	\leftline{ \scriptsize \it  }
	\title[]
	{Approximation by Certain Complex  Nevai  Operators : Theory and Applications}
	\maketitle

	\begin{center}
		{\bf Priyanka Majethiya,}
		{\bf Shivam Bajpeyi\footnote{corresponding author\\
				Email: shivambajpai1010@gmail.com, shivambajpeyi@amhd.svnit.ac.in }}

\vskip0.2in
Department of Mathematics, Sardar Vallabhbhai National Institute of Technology Surat, Gujarat-395007, India\\

\verb"  priyankamajethiya2000@gmail.com, shivambajpai1010@gmail.com" 
\end{center}

\begin{abstract}
	
	
The approximation of complex-valued functions is of fundamental importance as it generalizes classical approximation theory to the complex domain, providing a rigorous framework for amplitude and phase-dependent phenomena. In this paper, we study the Nevai operator, a concept formulated by the distinguished mathematician Paul G. Nevai. We propose a family of complex Nevai interpolation operators to approximate analytic as well as non-analytic complex-valued functions along with real-life application in image processing. In this direction, the first operator is constructed using Chebyshev polynomials of the first kind, namely complex generalized Nevai operators for approximating complex-valued continuous functions. We establish the approximation results for the proposed operators utilizing the notion of a modulus of continuity. To approximate not necessary continuous but integrable function, we define complex Kantorovich type Nevai operators and establish their boundedness and convergence. Furthermore, in order to approximate functions preserving higher derivatives, we introduce complex Hermite type Nevai operators and study their approximation capabilities using higher order of modulus of continuity. To validate the theoretical results, we provide numerical illustrations of approximation abilities of proposed family of complex Nevai operators.

	
\end{abstract}

\noindent \textbf{Keywords:}  Approximation in complex domain, generalized Nevai operator, Modulus of continuity, Peetre's K-functional, Image reconstruction\\

\vskip0.001in
\noindent \textbf{Subject Classification : 30E10, 41A35, 41A05, 41A81, 94A08  }

\section{Introduction}



The Nevai operator, first introduced by Paul G. Nevai, constitutes a fundamental tool in approximation theory and the study of orthogonal polynomials.  It provides a robust method to analyze polynomial approximations under very general conditions, without requiring restrictive assumptions on the underlying measure \cite{unit1}. 
Paul G. Nevai is a renowned mathematician celebrated for his significant contributions to the theory of orthogonal polynomials and approximation theory (see \cite{1979}). He has contributed to a deepen understanding of the behavior and convergence of orthogonal polynomials, particularly through his work on Lagrange interpolation, asymptotic formulas, and recurrence relations \cite{lem,1984,  mean1, mean2}.


For the orthonormal polynomial system \(\{u_n\}_{n=0}^\infty\)  associated with  a positive measure \(\mu\) on $\RR$, the \emph{Christoffel–Darboux kernel} is defined as
\[
K_n(v,t) = \sum_{k=0}^n u_k(v) u_k(t),
\]
and the \emph{Christoffel function} is given by
\[
\lambda_n(v) = \sum_{k=0}^n u_k^2(v) = \left(K_n(v,v)\right)^{-1}.
\]
The Christoffel–Darboux kernel and Christoffel function are fundamental tools in the theory of orthogonal polynomials. They are widely used in polynomial least squares approximation, solutions to the moment problem, approximating weight functions, and play a key role to understand the universality phenomena in random matrix theory (see \cite{lub2011,lub2009}).
In this direction, Nevai made pioneering contributions to the theory of orthogonal polynomials and Christoffel functions, particularly in exploring their asymptotic behaviour, and convergence properties \cite{geza}.  Nevai \cite{1979} introduced a  integral operators defined by  
\begin{equation}\label{original}
	F_n(f) = \frac{1}{K_n(v,v)}\int_{\operatorname{supp}(\mu)
	} K_n^2(v, t) f(t) \, d\mu(t),
\end{equation}
where $K_n$ denotes the Christoffel–Darboux kernels. These operators (\ref{original}) provide a powerful tool to study the approximation properties of Christoffel functions (see  \cite{1979}).
Moreover, these operators have broad applications in numerical analysis, spectral theory, and mathematical physics and serve as a cornerstone for advancing the theoretical insights and practical computations involving orthogonal polynomial systems (see \cite{lub2011, be}). 
Using the standard modulus of continuity along with weak asymptotic relations, Criscuolo \emph{et al.} \cite{cris} established pointwise error estimates for (\ref{original}). These operators are well-known as the \emph{Nevai operators} and have been investigated and extensively generalized over the years, see \cite{las,2002, 200, 1992,  2005}.


For any family of orthonormal polynomials $u_n(x)$, the corresponding discrete formulation of (\ref{original}) is given by 
\begin{equation}\label{n11}
	N_n(f) := \frac{
		\displaystyle \sum_{k=1}^n \frac{|\ell_{n,k}(x)|^s}{\lambda_{n,k}^{s/2}} f(x_k)
	}{
		\displaystyle \sum_{k=1}^n \frac{|\ell_{n,k}(x)|^s}{\lambda_{n,k}^{s/2}}
	}, \quad x \in [-1,1], \quad s \geq 2,
\end{equation}
where $\ell_{n,k}(x)= \displaystyle\frac{u_n(x)}{u_n'(x_k)(x-x_k)}$	are the fundamental
Lagrange  polynomials, $\lambda_{n,k}, k=1,2,...,n$  are the corresponding cotes numbers and $x_k$ are the zeros of $u_n(x)$. For \(s=2\), the operator (\ref{n11}) reduces to the operator which was originally introduced and studied by Nevai in \cite{1979}. It is worth noting that (\ref{n11}) can also be seen as a member of a wider class of linear, positive, and rational interpolatory operators as in \cite{cri,della}. Some significant advances in \cite{2002} include the rigorous formulation of direct and inverse theorems for weighted and unweighted uniform approximation by the Nevai operators. Furthermore, the author in \cite{della} proposed a modification of (\ref{n11})  and established  pointwise simultaneous approximation error estimates of Gopengauz–Teliakovskii type. 
Moreover, a uniform convergence result of Korovkin type for (\ref{n11}) has been proved in \cite{cri}. The Jackson-type estimates in weighted $L^p$-spaces along with the associated direct and converse theorems for modified Nevai operators are analyzed in \cite{200}.  Zhou  \cite{2005} introduced the Nevai–Durrmeyer operators and investigated their approximation behavior in weighted $L^p$-spaces. Some notable interpolation operators, namely  Hermite–Fejér interpolation  operator and Shepard operator can be constructed using the Nevai operator  (see \cite{cri}). In addition to their theoretical significance, the Nevai operators have found notable applications in the analysis of orthogonal polynomials on the unit circle, see \cite{unit1,unit2}.




In \cite{1992}, the author introduced and studied a new class of rational interpolation operators based on Nevai operator namely \emph{generalized Nevai interpolation operator} defined as
\begin{flalign}\label{n12}
	N_{n,p}(f) = \frac{\displaystyle \sum_{k=1}^n f(x_k) \lambda_k |K_n(x, x_k)|^p}{\displaystyle\sum_{k=1}^n \lambda_k |K_n(x, x_k)|^p}, \quad p \in [0, \infty),
\end{flalign}
where $K_n(x,x_k)$ is the kernel obtained from the Christoffel-Darboux formula and $\lambda_k$ is the Christoffel function. The operator (\ref{n12}) coincides with the Nevai type rational interpolation operators (\ref{n11}) in the case $p=2.$ 
The operators given in (\ref{n11}) and (\ref{n12}) have been playing a vital role in  approximation theory due of their linear, positive, and interpolatory nature.
\par 
Although substantial research has been devoted to the study of Nevai-type operators for approximating real-valued functions (see \cite{1992,cri,della, lub2011}), their potential for approximating complex-valued functions remains unexplored. In this paper, we construct and analyze generalized Nevai operators in the complex domain to approximate analytic as well as non-analytic functions. This construction utilizes Chebyshev polynomials of the first kind. The superiority of rational approximation over polynomials is well known, and elegant results on rational approximation are obtained by Herbert Stahl in  \cite{31}. Following this line of investigation, further studies have considered polynomial structures, especially Chebyshev polynomials. A comprehensive analysis of Hermite–Fejér interpolation operators on Chebyshev nodes is presented in \cite{cheb,che,ch}.
\par 

It is widely known that the approximation of complex-valued functions is fundamentally significant as such functions inherently capture oscillatory behavior and phase relationships that cannot be represented by real-valued functions alone. These are also critical in practical applications, including radar systems, synthetic aperture radar (SAR) imaging, and signal processing, providing a rigorous framework for modeling amplitude- and phase-dependent phenomena. In this direction, the problem of estimating errors in the weighted approximation of functions with Freud-type weights using entire functions was addressed in \cite{compl}. Recently, D. Yu \cite{dyu} proposed neural network interpolation operators activated by non-compactly supported functions, and established both direct and converse approximation results. Several useful constructions to approximate complex-valued functions can be observed in \cite{1,2,3, duman, jamma}.
\par 

In light of significance of Nevai operators and the long-standing interest for complex-valued function approximation, studying complex Nevai operators and its extensions appears noteworthy.


\subsection{Contributions}
The key contributions of the paper are listed below:
\begin{itemize}

	\item Inspired by the Nevai operator, a novel family of complex interpolation operators based on Nevai operators, is constructed using Chebyshev polynomials of the first kind. We propose certain class of complex Nevai operators for different class of target functions as follows. \\
	\begin{enumerate}[label=(\roman*)]
		\item  To approximate complex valued continuous functions, we introduce complex generalized Nevai operators (\ref{Nevai}) and analyze their approximation properties.
		\item 	In order to approximate not necessarily continuous but complex-valued $p-$integrable functions, we establish the complex Kantorovich type Nevai operators (\ref{kant}) and study their convergence behavior.
		\item  We construct and study the complex Hermite type Nevai operators (\ref{herm}) for approximating complex valued $r$-times differentiable  functions.\\
	\end{enumerate}
	
	
	
	\item Alongside the theoretical advancements, convergence is demonstrated using numerical examples that involves the approximation of the real and imaginary parts of complex-valued functions. In addition, the applicability of complex Kantorovich type Nevai operators is demonstrated in image reconstruction, where both the amplitude and phase are considered. The performance is measured by standard measures such as the structural similarity index measure (SSIM), peak signal-to-noise ratio (PSNR) and root mean square error (RMSE).

\end{itemize}

\subsection{Organization of the paper}
The paper is organized as follows:
\begin{itemize} 
	\item Section \ref{2} presents preliminary definitions and some important results which will be required for further analysis. \\
	\item In Section \ref{3},  we provide the construction and analysis of the complex generalized Nevai operators within continuous function space. Furthermore, in Section \ref{4}, we extend this framework to the complex Kantorovich type Nevai operators and study their approximation properties.\\
	\item  In Section \ref{5}, we define and study the complex Hermite type Nevai  operators for approximating $r$-times differentiable functions.\\
	\item   Section \ref{6} presents the illustration of approximation capabilities of the proposed family of complex Nevai operators and application in image-reconstruction.

\end{itemize}

\section{Preliminaries and notations}\label{2}
			 We use the notations $\mathbb{N}$, $\mathbb{Z}$, $\mathbb{R}$ and $\mathbb{C}$ to represent the set of natural numbers, integers, real numbers, and complex numbers respectively.  	Here we denote $
		 X = [-1,1] \times [-1,1] = \{ \mathscr{z} = x + iy \in \mathbb{C} : x, y \in [-1,1]\},$
		 where $i^2 = -1$, $x = \Re(\mathscr{z})$ and $y = \Im(\mathscr{z})$.  Here $|\cdot|_2$ denotes the usual Euclidean norm on the set $X$. 	We denote the space of all complex-valued continuous  functions  by $C(X)$ equipped with the norm $\|f\|_{\infty} :=\sup_{t \in X } |f(t)|.$ The space of all complex-valued  absolutely  continuous functions is referred to as $AC(X)$. 
		   The notion of modulus of continuity for $f \in C(X)$ is defined as (\cite{lorentz1963degree})
		 	\[
		 \omega(f,\delta) = \sup_{\substack{x,y \in X \\ |x - y|_2 \leq \delta}} |f(x) - f(y)|.
		 \]
		 The space of all complex-valued  $r$-times differentiable  functions is denoted by $C^r(X)$. Moreover, the modulus of continuity for  $f \in C^r(X)$ is given by (\cite{lorentz1963degree})
			\[
			\omega(f^r,\delta) = \sup_{\substack{x,y \in X  \\ |x - y|_2 \leq \delta}} |f^r(x) - f^r(y)|.
			\]
			It is used to describe smoothness and approximation properties.
			It is worth noting that the following statements hold for $\lambda>0$:
			\begin{equation}\label{omega}
				\omega(f, \lambda \delta) \leq (1 + \lambda)\hspace{2pt}\omega(f,\delta),
			\end{equation}
			and
				\begin{equation}\label{ome}
				\omega(f^r, \lambda \delta) \leq (1 + \lambda)\hspace{2pt}\omega(f^r,\delta).
			\end{equation}
			Moreover, the set of all complex-valued  $p-$integrable   function is denoted by $L^p(X)$, for $1 \leq p < \infty$, which consists of equivalence classes of measurable functions $f: X \to \mathbb{C}$ satisfying  $
			\int_{X} |f(x)|^p \, dx < \infty.$  
			The corresponding norm is given by  
			\[
			\|f\|_{p} = \left(\int_{X} |f(x)|^p \, dx\right)^{1/p}.
			\] 
			 The modulus of continuity for $f \in L^p(X)$ is defined as (\cite{lorentz1963degree})
			\[
			\omega(f,\delta)_p = \sup_{|h|\leq \delta} \|f(x+h) - f(x)\|_p,
			\]
			 and the Peetre's K-functional for $f \in L^p(X)$ is defined as (\cite{joh})
			\[
			K(f,t)_p := \inf_{D^{\alpha} h \in C(X)} \{\|f - h\|_p + t \sup_{|\alpha|=1} \|D^{\alpha} h\|_p\}
			\]
			where $|\alpha|=\alpha_1+...+\alpha_s=1$ with $\alpha_i \in  \{  0 , 1 \}$ for $i=1,2,...,s$, $D^{\alpha}=\frac{\partial^{\alpha}}{\partial x_1^{\alpha_1}\cdots\partial x_s^{\alpha_s}}$.

				\subsection{Chebyshev polynomials}
			Let \( w(x) = (1 - x)^\alpha (1 + x)^\beta \; (-1 < \alpha, \beta < 1) \) denote the Jacobi weight function. In the rest of the paper we will be dealing with the special case of Jacobi polynomials, namely the Chebyshev polynomials of the first kind. The   Chebyshev polynomials of the first kind for $\alpha=\beta=-1/2$ are given by
			\[
			P_n(x)  = \cos\!\big(n \arccos x\big).
			\]
			The polynomials $	P_n(x)$ satisfy the following relation

\[
\int_{-1}^{1} \frac{P_m(x)\,P_n(x)}{\sqrt{1-x^2}} \, dx
=
\begin{cases}
	\pi, & m=n=0, \\[6pt]
	\frac{\pi}{2}, & m=n\geq 1, \\[6pt]
	0, & m \ne n.
\end{cases}
\]

The orthonormal version of the Chebyshev polynomials of the first kind $P_n(x)$, can be written as
\[
T_0(x)=\frac{1}{\sqrt{\pi}},\qquad
T_n(x)=\sqrt{\frac{2}{\pi}}\,P_n(x)\ (n\ge1),
\]
and
\[
\int_{-1}^{1}\frac{T_n(x)\hspace{2pt}T_m(x)}{\sqrt{1-x^2}}\,dx=\delta_{mn},
\]
			where \(\delta_{mn}\) denotes the Kronecker's delta symbol.
		
		\subsubsection{Christoffel functions}
				The Christoffel functions \(\lambda_n(v)\) are closely connected with  the cotes numbers \(\lambda_k = \lambda_n(v_k)\), where \(\{v_k\}\) denotes the zeros of \(T_n(v)\) arranged in the increasing order, i.e,
				$v_n < v_{n-1} < \cdots < v_1.$
			The reproducing kernel function is written as 
			\[
			K_n(v,t) =  \sum_{k=0}^n T_k(v)\hspace{2pt} T_k(t),  \]
			 by using the Christoffel-Darboux formula, which is given as
			\begin{equation}\label{chris}
				K_n(v,t) =  \frac{T_n(v)\hspace{2pt} T_{n-1}(t) - T_{n-1}(v) \hspace{2pt}T_n(t)}{v - t}.
			\end{equation}
  It can be observed that
	\begin{equation}\label{a}
	\sum_{k=1}^n \lambda_k = M_1 , \quad 	\sum_{m=1}^n \lambda_m =  M_2,
	\end{equation}
	for  some  constant $M_1,M_2>0.$
	
	 \subsubsection{Interpolation with Chevyshev nodes}
 For a given $n \in \mathbb{N}$, we consider the following sample points for the set $X$ :
		\[
		\mathscr{z}_{k,m} = x_k + i y_m, \quad k, m \in \{  1, \ldots, n \},
		\]
		where $x_k$ and $y_m$ are zeros of Chebyshev polynomials of the first kind  $T_n(x)$ and $T_n(y),$ respectively i.e., $ x_k = \cos\!\left(\frac{(2k-1)\pi}{2n}\right), $ and $ y_m = \cos\!\left(\frac{(2m-1)\pi}{2n}\right).$

		Now we are ready to define and analyze the proposed family of complex Nevai operators.
		
			\section{ Approximation by Complex generalized Nevai Operators}\label{3}
		Let $f$ be a complex-valued continuous function defined on $X$. Then, for a positive real number $\lambda_{k,m}$ (defined in Section \ref{2}), we define the complex generalized   Nevai  operators as
		\begin{equation}\label{Nevai}
			N_{n,s}(f, \mathscr{z}) = \frac{\displaystyle \sum_{k=1}^{n}\sum_{m=1}^{n} \lambda_{k,m} \left| K_n(\mathscr{z}, \mathscr{z}_{k,m}) \right|^{s} f(\mathscr{z}_{k,m})}{ \displaystyle\sum_{k=1}^{n}\sum_{m=1}^{n} \lambda_{k,m} \left| K_n(\mathscr{z}, \mathscr{z}_{k,m}) \right| ^{s} }, \quad s \in [0,\infty),
		\end{equation}
where \( K_n(\mathscr{z},  \mathscr{z}_{k,m}) = K_n(x, x_k) \hspace{2pt} K_n(y, y_m). \)
		 One can observe that the complex generalized Nevai operator (\ref{Nevai}) is a positive linear operator and interpolates at the sample points $\mathscr{z}_{k,m}$, i.e.,
		\[
		N_{n,s}(f, \mathscr{z}_{k,m}) = f(\mathscr{z}_{k,m}) \quad \text{for} \quad k, m \in \{ 1, \ldots, n\}.
		\]

			Before proving the main result of this section, we first present the following lemma.
		
			\begin{lemma}\label{lem 3.1} \cite{lem}
			 Let $x_k = \cos \theta_k$ and $y_m = \cos \phi_m$ , with $0 \leq \theta_k, \phi_m \leq \pi$. Then
			\begin{equation}\label{2.6}
				\theta_{k+1} - \theta_{k} \sim \frac{1}{n}, \qquad \phi_{m+1} - \phi_{m} \sim \frac{1}{n}, \qquad 0 \leq k,m \leq n,
			\end{equation}
			and
			\begin{equation}\label{2.7}
				\lambda_n(\mathscr{z})=	\lambda_n(x) 	\lambda_n(y) \sim 
				\begin{cases}
					\dfrac{1}{n^2} , & |x| \leq 1 - n^{-2} , |y| \leq 1 - n^{-2}, \\[1.2ex]
					\dfrac{1}{n^2}, & 1 - n^{-2} \leq x \leq 1 , 1 - n^{-2} \leq y \leq 1  , \\[1.2ex]
					\dfrac{1}{2n^4-n^{-2}}, & -1 \leq x \leq -1 + n^{-2} , -1 \leq y \leq -1 + n^{-2}.
				\end{cases}
			\end{equation}
				Moreover,
			\begin{equation}\label{2.8}
				\lambda_{k,m}=	\lambda_{k}	\lambda_{m}  \sim  \frac{1}{n^2}, 
				\qquad k,m = 1, \dots, n,
			\end{equation}
			
			
			
		
			\begin{equation}\label{2.9}
				|T_{n-1}(x_{k})| \sim 1, \quad \text{and  }\quad
				|T_{n-1}(y_{m})| \sim 1,
			\end{equation}
			\begin{equation}\label{2.10}
			1-	x_n \sim 1+x_1 \sim \frac{1}{n^2}, \quad\text{  and  }\quad
			1-	y_n \sim 1+y_1 \sim \frac{1}{n^2}.
			\end{equation}
			\end{lemma}

			\begin{lemma}\label{lem 3.2} \cite{1984}
				For any Lagrange  polynomial $\ell_k(x)$ and $ \ell_k(y)$, we have
				\begin{enumerate}
					\item 
						$	|\ell_k(x)|  \sim  \displaystyle\frac{|T_n(x)|\sqrt{1-x_k^2}}{n|x-x_k|}, \quad |\ell_m(y)|  \sim  \displaystyle\frac{|T_n(y)|\sqrt{1-y_m^2}}{n|y-y_m|} .$
			\item 	$|x-x_k| \sim \displaystyle \frac{|p-k|}{n} \sqrt{1-x^2}+\frac{(p-k)^2}{n^2},  \quad k \neq p, \qquad |y-y_m| \sim \frac{|q-m|}{n} \sqrt{1-y^2}+\frac{(q-m)^2}{n^2},  \quad m \neq q,$
				where $x_p$ is the closest zero to $x,$ and  $y_q$ is the closest zero to $y.$
					\item 
				$\displaystyle \frac{1-x_k^2}{n^2(x-x_k)^2}\leq \displaystyle \frac{1}{(k-p)^2},  \quad k \neq p, \qquad \displaystyle\frac{1-y_m^2}{n^2(y-y_m)^2}\leq \displaystyle \frac{1}{(m-q)^2},  \quad m \neq q$.
					\end{enumerate}
			\end{lemma}
				\begin{remark}\label{rem 1}
				\cite{freud}	Let $\ell_k(x)$ be the fundamental Lagrange  polynomial. Using the well-known formula
					 \begin{equation}\label{lkk}
						\ell_k(x)=\lambda_k K_n(x,x_k)
					\end{equation} 
					 we can write
				 $$\lambda_k | K_n(x,x_k)|^s= \displaystyle \frac{|\ell_k(x)|^s}{\lambda_k^{s-1}}.$$
			\end{remark}

		\begin{lemma}\label{lem 3.33} \cite{1992}
		Let $\mathscr{z}_{p,q}$ be the  nearest node to $\mathscr{z}$. Then for $\mathscr{z} \in [x_p,x_{p+1}] \times [y_q,y_{q+1}]$, we have
			\begin{equation}\label{3.10}
					w(x_{p})\ \asymp\ w(x)\ \asymp\ w(x_{p+1}) \text{ and } 	\lambda_{p}\ \asymp\ \lambda_n(x)\ \asymp\ \lambda_{p+1},
			\end{equation}
			and
			\begin{equation}\label{3.111}
				 	w(y_{q})\ \asymp\ w(y)\ \asymp\ w(y_{q+1})  \text{  and  } 	\lambda_{q}\ \asymp\ \lambda_n(y)\ \asymp\ \lambda_{q+1}.
			\end{equation}
		\end{lemma}

				\begin{lemma} \label{lem 3.3}
					For any $ \mathscr{z} \in X$, the following assertions holds
				\[
				B_{n,s}(\mathscr{z}) = \sum_{k=1}^{n} \sum_{m=1}^{n} \lambda_{k,m} |K_n(\mathscr{z},\mathscr{z}_{k,m})|^{s} \geq c \left(K_n(\mathscr{z},\mathscr{z})\right)^{s-1}, \quad  s \geq 1 .
				\]
				\end{lemma}
				\begin{proof}
					We prove the result by splitting in  following cases:
					\paragraph{\textbf{Case 1}}
					For any $\mathscr{z} \in [x_n,x_1]\times [y_n,y_1]$, there exist  indices $p,q$ such that $\mathscr{z} \in [x_{p+1}, x_{p}) \times [y_{q+1},y_{q})$. Then we have
				\begin{equation}\label{l}
				|l_p(x)| + |l_{p+1}(x)| \geq 1, \text{ and } 	|l_q(y)| + |l_{q+1}(y)| \geq 1, \text{ for }  \mathscr{z} \in [x_{p+1}, x_{p}) \times [y_{q+1},y_{q}).
				\end{equation}
			In view of triangle inequality and  (\ref{lkk})-(\ref{l}), we have
				\begin{flalign}\label{sp}
	1 &\leq \lambda_{p,q} |K_n(\mathscr{z},\mathscr{z}_{p,q})| + \lambda_{p+1,q+1}  |K_n(\mathscr{z},\mathscr{z}_{p+1,q+1})| + \lambda_{p,q+1} |K_n(\mathscr{z},\mathscr{z}_{p,q+1})| + \lambda_{p+1,q}  |K_n(\mathscr{z},\mathscr{z}_{p+1,q})|\nonumber\\
&\leq c\lambda_{p,q}  \left[ |K_n(\mathscr{z},\mathscr{z}_{p,q})|+ |K_n(\mathscr{z},\mathscr{z}_{p+1,q+1})| + |K_n(\mathscr{z},\mathscr{z}_{p+1,q})|+ |K_n(\mathscr{z},\mathscr{z}_{p,q+1})| \right] \nonumber\\
				&\leq  c\lambda_{p,q} 2^{2-\frac{2}{s}}\left[ |K_n(\mathscr{z},\mathscr{z}_{p,q})|^s + |K_n(\mathscr{z},\mathscr{z}_{p+1,q+1})|^s+ |K_n(\mathscr{z},\mathscr{z}_{p+1,q})|^s+ |K_n(\mathscr{z},\mathscr{z}_{p,q+1})|^s\right]^{\frac{1}{s}}.
				\end{flalign}
				From (\ref{sp}) and Lemma \ref{lem 3.33}, we obtain
				\begin{flalign*}
				B_{n,s}(\mathscr{z}) &\geq c \lambda_{p,q} \left[ |K_n(\mathscr{z},\mathscr{z}_{p,q})|^s + |K_n(\mathscr{z},\mathscr{z}_{p+1,q+1})|^s + |K_n(\mathscr{z},\mathscr{z}_{p+1,q})|^s+ |K_n(\mathscr{z},\mathscr{z}_{p,q+1})|^s \right]\\
				&\geq c\lambda_{p,q} \left[\lambda_{p,q} 2^{2-\frac{2}{s}}\right]^{-s}\\
				& \geq c \left(\lambda_n(\mathscr{z}) \right)^{1-s}\\
				&= c \left(K_n(\mathscr{z},\mathscr{z})\right)^{s-1}.
				\end{flalign*}
					\paragraph{\textbf{Case 2}}
				For $\mathscr{z} \in (x_1,1]\times (y_1,1]$,  we have $
				|l_1(x)| \geq 1 \text{ and } |l_1(y)| \geq 1.$
					Using (\ref{2.7}) and (\ref{2.8}) we may write
				\begin{flalign*}
				B_{n,s}(\mathscr{z}) 
				\geq  \lambda_{1,1} |K_n(\mathscr{z},\mathscr{z}_{1,1})|^s \geq \lambda_{1,1}^{1-s}
				 \geq c_1c_2\left(\frac{1}{n^2}\right)^{1-s}
				&\geq c_1c_2\lambda_n^{1-s}(\mathscr{z}) = c\left(K_n(\mathscr{z},\mathscr{z})\right)^{s-1}.
				\end{flalign*}
			
				\paragraph{\textbf{Case 3}}
				For $\mathscr{z} \in [-1,x_{n}) \times [-1,y_{n})$, we deduce the following by similar arguments:
				\[
				B_{n,s}(\mathscr{z}) \geq c \left(K_n(\mathscr{z},\mathscr{z})\right)^{s-1}.
				\]
				This proves the result.
				\end{proof}

			\begin{lemma} \label{lem 3.7}
				 For \(1 < s \leq 2\), the following inequality holds:
			\begin{flalign*}
				\sum_{k=1}^n 	\sum_{m=1}^n|\mathscr{z} - \mathscr{z}_{k,m}|_2 \hspace{2pt}\lambda_{k,m} |K_n(\mathscr{z}, \mathscr{z}_{k,m})|^s \leq  \left[ 1 + \ln n \right] \left(B_2|T_n(x)|^s+ B_1|T_n(y)|^s \right).
			\end{flalign*}
			\end{lemma}
			
%
			\begin{proof}
First we estimate
			\begin{flalign}\label{s1s2}
			\sum_{k=1}^{n} \sum_{m=1}^{n} |\mathscr{z} - \mathscr{z}_{k,m}|_2 \hspace{2pt} \lambda_{k,m} |K_n(\mathscr{z}, \mathscr{z}_{k,m})|^s 
			&=\sum_{k=1}^n |x - x_{k}| \lambda_k  |K_n(x,x_k)|^s \sum_{m=1}^n  \lambda_m  |K_n(y, y_{m})|^s \nonumber\\ 
			&+\sum_{m=1}^n |y - y_{m}| \lambda_m |K_n(y, y_{m})|^s  \sum_{k=1}^n \lambda_k |K_n(x,x_k)|^s\nonumber\\
			&=\sum_{k=1}^n |x - x_{k}| \lambda_k  |K_n(x,x_k)|^s A_m \nonumber\\ 
			&+\sum_{m=1}^n |y - y_{m}| \lambda_m |K_n(y, y_{m})|^s  A_k \nonumber\\
			&=S_1+S_2,
			\end{flalign}
			where 
			$$A_m :=  \sum_{m=1}^n  \lambda_m  |K_n(y, y_{m})|^s, \quad A_k :=\sum_{k=1}^n \lambda_k |K_n(x,x_k)|^s.$$
			To estimate $A_m$ and $A_k$, we use  Hölder's inequality as follows
			\[
			A_k 
			= \sum_{k=1}^n \lambda_k \, |K_n(x,x_k)|^s
			\leq 
			\left( \sum_{k=1}^n \lambda_k |K_n(x,x_k)|^2 \right)^{s/2}
			\left( \sum_{k=1}^n \lambda_k \right)^{1 - \tfrac{s}{2}}.
			\]
			Using (\ref{a}), we obtain
			\begin{flalign*}
			A_k 
			 \leq M_1 \left(\frac{1}{M_1\lambda_n(x)}\right)^{\frac{s}{2}} :=B_1,
			\end{flalign*}
		for some constant $B_1>0$. Similarly for $A_m$, we get
			\begin{equation*}
					A_m 
				\leq M_2 \left(\frac{1}{M_2\lambda_n(y)}\right)^{\frac{s}{2}} :=B_2,
			\end{equation*}
		for some constant $B_2>0$.	In view of (\ref{chris}) and (\ref{2.9}) we can write
			\begin{flalign}\label{s1}
				S_1 &= B_2\sum_{k=1}^n |\ell_k(x)| |x - x_k|^{2-s} \left|  T_{n-1}(x_k) T_n(x) \right|^{s-1} \nonumber\\
					&\leq  B_2|T_n(x)|^{s-1} \sum_{k=1}^n |\ell_k(x)| .
			\end{flalign}
				From \cite{natan} we can write
			\begin{equation}\label{3.11}
	\left| \sum_{k=1}^n \ell_k(x) \right| \sim |T_n(x)|[1 +  \ln n].
			\end{equation}
				Using (\ref{s1}) and (\ref{3.11}), we obtain
			\begin{flalign}\label{f1}
			S_1 & \leq B_2 |T_n(x)|^s[1 +  \ln n].
			\end{flalign}
			Similarly,
			\begin{flalign}\label{f2}
				S_2 \leq B_1 |T_n(y)|^s[1 +  \ln n].
			\end{flalign}
		On combining (\ref{s1s2}), (\ref{f1}) and (\ref{f2}), we obtain
			\begin{flalign*}
					\sum_{k=1}^n	\sum_{m=1}^n |\mathscr{z} - \mathscr{z}_{k,m}|_2 \hspace{2pt}\lambda_{k,m} |K_n(\mathscr{z}, \mathscr{z}_{k,m})|^s 
			&\leq B_2|T_n(x)|^s \left[ 1 +  \ln n \right]
			+ B_1 |T_n(y)|^s \left[ 1 + \ln n \right]\\
			& \leq  \left[ 1 + \ln n \right] \left(B_2|T_n(x)|^s+ B_1|T_n(y)|^s \right).
			\end{flalign*}
			\end{proof}
			
We are now in a position to prove the main result of this section, that is the quantitative approximation result for $f \in C(X)$.
		
		\begin{thm} \label{thm 3.8}
			Let $f \in C(X)$, $1 < s \leq 2$.
			Then we have
			\begin{equation}\label{main}
				\left| N_{n,s}(f,\mathscr{z}) - f(\mathscr{z}) \right|
				 \leq 
				 2\omega\left(f, \, \frac{\lambda_n^{1-s}(\mathscr{z})\left[ 1 + \ln n \right]}{\left(B_2|T_n(x)|^s+ B_1|T_n(y)|^s \right)}\right).
			\end{equation}
		\end{thm}
		\begin{proof}
				From (\ref{Nevai}), we have
			\[
			\left| N_{n,s}(f,\mathscr{z}) - f(\mathscr{z}) \right|
			\leq
			\frac{\displaystyle \sum_{k=1}^{n} \sum_{m=1}^{n} \left| f(\mathscr{z}_{k,m}) - f(\mathscr{z})\right|
				\, \lambda_{k,m} |K_n(\mathscr{z},\mathscr{z}_{k,m})|^s}
			{\displaystyle \sum_{k=1}^{n} \sum_{m=1}^{n} \lambda_{k,m} |K_n(\mathscr{z},\mathscr{z}_{k,m})|^s}.
			\]
				Using property (\ref{omega}) and Lemma \ref{lem 3.3}, we get
			\begin{flalign*}
		\left| N_{n,s}(f,\mathscr{z}) - f(\mathscr{z}) \right|
			&\leq 	\frac{\displaystyle \sum_{k=1}^{n} \sum_{m=1}^{n} \omega( f, |\mathscr{z}_{k,m} - \mathscr{z}|_2 \hspace{2pt})
				\, \lambda_{k,m} |K_n(\mathscr{z},\mathscr{z}_{k,m})|^s}
			{\displaystyle \sum_{k=1}^{n} \sum_{m=1}^{n} \lambda_{k,m} |K_n(\mathscr{z},\mathscr{z}_{k,m})|^s}\\
				&\leq \omega( f, \delta)	\frac{\displaystyle \sum_{k=1}^{n} \sum_{m=1}^{n} \left( 1+\frac{|\mathscr{z}-\mathscr{z}_{k,m}|_2}{\delta}\right)
				\, \lambda_{k,m} |K_n(\mathscr{z},\mathscr{z}_{k,m})|^s}
			{\displaystyle \sum_{k=1}^{n} \sum_{m=1}^{n} \lambda_{k,m} |K_n(\mathscr{z},\mathscr{z}_{k,m})|^s}\\
				&\leq \omega( f, \delta)	\left\lbrace 1+ \displaystyle \frac{\lambda_n^{1-s}(\mathscr{z})}{\delta}\sum_{k=1}^{n} \sum_{m=1}^{n}  |\mathscr{z}-\mathscr{z}_{k,m}|_2
			\, \lambda_{k,m} |K_n(\mathscr{z},\mathscr{z}_{k,m})|^s\right\rbrace .
			\end{flalign*}
			Now by Lemma \ref{lem 3.7}, we obtain 
			\begin{flalign*}
		\left| N_{n,s}(f,\mathscr{z}) - f(\mathscr{z}) \right|
		&	\leq
			\omega(f,\delta)\left\lbrace 
			1 + \frac{\lambda_n^{1-s}(\mathscr{z})\left[ 1 + \ln n \right]}{\delta} \left(B_2|T_n(x)|^s+ B_1|T_n(y)|^s \right)\right\rbrace.
			\end{flalign*}
				By choosing
			\[
			\delta = \frac{\lambda_n^{1-s}(\mathscr{z})\left[ 1 + \ln n \right]}{\left(B_2|T_n(x)|^s+ B_1|T_n(y)|^s \right)},
			\]
			we get the required estimate.
		\end{proof}

		\section{ Approximation by Complex Kantorovich  type Nevai Operators}\label{4}
		For any complex-valued $p-$integrable  function $f : X \rightarrow \CC$, the family of complex  Kantorovich type Nevai  operators  for $n \in \NN$ is defined as 
	\begin{flalign}\label{kant}
		 K_{n,s}(f, \mathscr{z}) 
		&= n^2\displaystyle \sum_{k=-n}^{n}\sum_{m=-n}^{n}L_{k,m,n}(\mathscr{z})	\int_{\frac{k}{n}}^{\frac{k+1}{n}} \int_{\frac{m}{n}}^{\frac{m+1}{n}} f(u,v)dvdu
	\end{flalign}
	where 
	\begin{equation}\label{lk}
		L_{k,m,n}(\mathscr{z}) :=\frac{ \lambda_{k,m} \left| K_n(\mathscr{z}, \mathscr{z}_{k,m}) \right|^s }{\displaystyle\sum_{k=-n}^{n}\sum_{m=-n}^{n} \lambda_{k,m} \left| K_n(\mathscr{z}, \mathscr{z}_{k,m}) \right|^s }, \quad s \in [0,\infty).
	\end{equation}
	
		In order to prove the convergence of above family (\ref{kant}) in the framework of  Lebesgue space $ L^{p}(X)$, we  establish the following results.

	\begin{lemma}\cite{joh}\label{e4}
	For every \( f \in L^p(X) \), there exist \( C_1>0 \) and \( C_2 >0\) such that the following equivalence holds:  
	\begin{equation} \label{o1}
		C_1 \hspace{1pt} \omega(f,t)_p \leq K(f,t)_p \leq C_2 \hspace{1pt}  \omega(f,t)_p, \quad t \in X.
	\end{equation}
\end{lemma}

	\begin{lemma}\label{lem 4.1}
		For any $\mathscr{z} \in [x_{p+1},x_p] \times [y_{q+1},y_q]$, $0 \leq p,q,m,k \leq n$, there holds 
	\begin{equation}\label{4.2}
		L_{k,m,n}(\mathscr{z}) \leq	D_s \left(|k-p| +1 \right)^{-s}\left(|m-q| +1 \right)^{-s},\quad s>1. 
	\end{equation}	
	\end{lemma}
	
	\begin{proof}
		For $|k-p|\leq 1$ and $|m-q| \leq 1$,  (\ref{4.2}) is obvious,  as $L_{k,m,n}(\mathscr{z}) \leq 4$. Now we consider the case  of  $min(|k-p|,|m-q|)>1$. In view of (\ref{lk}), we obtain
		\begin{flalign*}
			L_{k,m,n}(\mathscr{z})=\frac{ \lambda_{k,m} \left| K_n(\mathscr{z}, \mathscr{z}_{k,m}) \right|^{s} }{\displaystyle\sum_{k=-n}^{n}\sum_{m=-n}^{n} \lambda_{k,m} \left| K_n(\mathscr{z}, \mathscr{z}_{k,m}) \right|^{s} }
			& \leq \frac{ \lambda_{k,m} \left| K_n(\mathscr{z}, \mathscr{z}_{k,m}) \right|^{s} }{ \lambda_{p,q} \left| K_n(\mathscr{z}, \mathscr{z}_{p,q}) \right|^{s} }\\
			& = \frac{ \lambda_{k} \left| K_n(x,x_k) \right|^{s}  \lambda_{m}  \left| K_n(y, y_m) \right|^{s}}{ \lambda_{p} \left| K_n(x, x_{p}) \right|^{s} \lambda_{q} \left| K_n(y, y_{q}) \right|^{s}}.
		\end{flalign*}
Using Remark \ref{rem 1}  and Lemma \ref{lem 3.2}, we can write
		\begin{flalign*}
				L_{k,m,n}(\mathscr{z})\leq  \displaystyle \frac{\displaystyle \frac{|\ell_k(x)|^s}{\lambda_k^{s-1}}  \frac{|\ell_m(y)|^s}{\lambda_m^{s-1}}}{\displaystyle\frac{|\ell_p(x)|^s}{\lambda_p^{s-1}}  \displaystyle\frac{|\ell_q(y)|^s}{\lambda_q^{s-1}}}
			&\leq \frac{ \left(\frac{|T_n(x)|\sqrt{1-x_k^2}}{n|x-x_k|}\right)^s
				\left(\frac{|T_n(y)|\sqrt{1-y_m^2}}{n|y-y_m|}\right)^s }
			{ \left(\frac{|T_n(x)|\sqrt{1-x_p^2}}{n|x-x_p|}\right)^s 
				\left(\frac{|T_n(y)|\sqrt{1-y_q^2}}{n|y-y_q|}\right)^s } \\
				&\leq \frac{\left(1-x_k^2\right)^{s/2}\left(1-y_m^2\right)^{s/2}|x-x_p|^s|y-y_q|^s}{\left(1-x_p^2\right)^{s/2}\left(1-y_q^2\right)^{s/2}|x-x_k|^s|y-y_m|^s}\\
					&\leq \frac{n^{2s} |x-x_p|^s|y-y_q|^s}{\left(1-x_p^2\right)^{s/2}\left(1-y_q^2\right)^{s/2}|k-p|^s|m-q|^s}\\
					& \leq D_s \left(|k-p| +1 \right)^{-s}\left(|m-q| +1 \right)^{-s}.
				\end{flalign*}
	\end{proof}
	\begin{note}
	For any fixed $\mathscr{z} = x + i y \in X$, and $\mathscr{w} = u+iv \in X$, we define the function	$\psi$ on $X$ by 
\begin{equation}\label{psi}
	\psi(\mathscr{z},\mathscr{w}) := |\mathscr{w} - \mathscr{z}|_2 = \sqrt{(u-x)^{2} + (v-y)^{2}}.
\end{equation}	
	\end{note}
		\begin{lemma}\label{lem 4.2}
		Let \( \mathscr{z} \in X \) and $\psi$ as defined in (\ref{psi}). Then we have
		\[
		K_{n,s}(\psi, \mathscr{z}) =
		\begin{cases}
				\mathcal{O}\left( \frac{1}{n} \right), & \text{if } s > 2 \\[1.5ex]
			\mathcal{O}\left( \frac{\log n}{n} \right), & \text{if } s = 2.
		\end{cases}
		\]
	\end{lemma}
	
	\begin{proof}
	For any \( \mathscr{z} = x + i y \in X \), there exist \( p,q \in \{0, 1, \ldots, n\} \) such that
	\( x \in \left[ \frac{p}{n}, \frac{p+1}{n} \right] \) and \( y \in \left[ \frac{q}{n}, \frac{q+1}{n} \right] \).
	Using (\ref{kant}) and Lemma \ref{lem 4.1}, we obtain 
	\begin{align*}
		K_{n,s}(\psi, \mathscr{z})
		&= n^2 \displaystyle\sum_{k=-n}^{n}\sum_{m=-n}^{n} 	L_{k,m,n}(\mathscr{z})
		\int_{\frac{k}{n}}^{\frac{k+1}{n}} \int_{\frac{m}{n}}^{\frac{m+1}{n}}
		\sqrt{(u-x)^2 + (v-y)^2}\, dt\,ds \\
		&\leq n^2 \displaystyle\sum_{k=-n}^{n}\sum_{m=-n}^{n} 	L_{k,m,n}(\mathscr{z})
		\frac{\sqrt{(|k-p|+1)^2 + (|m-q|+1)^2}}{n^3} \\
		&\leq \frac{1}{n} \displaystyle\sum_{k=-n}^{n}\sum_{m=-n}^{n} 	\left(|k-p| +1 \right)^{-s}\left(|m-q| +1 \right)^{-s}
		\sqrt{(|k-p|+1)^2 + (|m-q|+1)^2} \\
		&\leq \frac{1}{n} \displaystyle\sum_{k=-n}^{n}\sum_{m=-n}^{n} 	\left(|k-p| +1 \right)^{-s}\left(|m-q| +1 \right)^{-s}
		\left(	(|k-p|+1) + (|m-q|+1) \right) \\
		&\leq \frac{1}{n} \sum_{k=-n}^n 	\left(|k-p| +1 \right)^{-s+1}\sum_{m=-n}^n\left(|m-q| +1 \right)^{-s}\\
		&\hspace{4cm} +  \frac{1}{n}\sum_{k=-n}^n	\left(|k-p| +1 \right)^{-s} \sum_{m=-n}^n\left(|m-q| +1 \right)^{-s+1}\\
		&:= \frac{1}{n} \left\lbrace H^p_{s-1} H^q_{s}+  H^p_{s} H^q_{s-1}\right\rbrace. 
	\end{align*}
	To simplify, we consider  $H_s:=max\left\lbrace  H^p_{s}, H^q_{s}\right\rbrace $.
	Hence
	$$K_{n,s}(\psi, \mathscr{z}) \leq \frac{2}{n}H_{s-1}H_{s}.$$
	Therefore, we see that
	\[
	K_{n,s}(\psi, \mathscr{z}) =
	\begin{cases}
			\mathcal{O}\left( \frac{1}{n} \right), & \text{if } s > 2 \\[1.5ex]
			\mathcal{O}\left( \frac{\log n}{n} \right), & \text{if } s = 2
	\end{cases}
	\]
	which completes the proof.
	
	\end{proof}
	
		The following theorem addresses the convergence of (\ref{kant}) in \( C(X) \).
	\begin{thm}\label{thm 4.3}
		For every $f \in C(X)$ and $s \geq 2$, we have
		\[
		\lim_{n \to \infty} \left\|  K_{n,s}f - f \right\|_{\infty} = 0.
		\]
	\end{thm}
	
	\begin{proof}
	 Let $f \in C(X)$. Using the uniform continuity of $f$ and (\ref{kant}), we deduce that
			\begin{flalign*}
				| K_{n,s}f- f| &\leq n^2 \displaystyle\sum_{k=-n}^{n}\sum_{m=-n}^{n}	L_{k,m,n}(\mathscr{z})
			 \int_{\frac{k}{n}}^{\frac{k+1}{n}} \int_{\frac{m}{n}}^{\frac{m+1}{n}} |f(u,v) - f(x,y)| \, dv \, du \\
				& \leq n^2 \displaystyle\sum_{k=-n}^{n}\sum_{m=-n}^{n} 	L_{k,m,n}(\mathscr{z})
				  \int_{\frac{k}{n}}^{\frac{k+1}{n}} \int_{\frac{m}{n}}^{\frac{m+1}{n}} \left( \epsilon + \frac{2N}{\delta} \sqrt{(u-x)^2 + (v-y)^2} \right) dv du \\
				& = \epsilon + \frac{2N}{\delta} K_{n,s}(\psi, \mathscr{z}),
			\end{flalign*}
			where $N := \|f\|_{\infty}$. This completes the proof by applying Lemma \ref{lem 4.2}.
		\end{proof}

	Since the convergence of $(K_{n,s})$ on $C(X)$ implies convergence in $L^p(X)$, the subsequent result is an immediate consequence of Theorem \ref{thm 4.3}.
	
	\begin{thm}\label{thm 4.4}
		 Let $f \in C(X)$ and $s \geq 2$. Then we have
		\[
		\lim_{n \to \infty} \left\|  K_{n,s}f - f \right\|_p = 0.
		\]
	\end{thm}
	
	In the following result we prove that (\ref{kant}) is bounded in $L^p(X).$
	\begin{lemma}
		\label{lem 4.5}
		For   $f \in L^p(X)$, where $1 \leq p < \infty$ and $s \geq 2$, there holds
	$$	\|K_{n,s}f\|_p\leq C\|f\|_p,$$
	for some $C>0.$
	\end{lemma}
	
	\begin{proof}
		In view of Jensen's inequality, we get
		\begin{flalign}\label{e1}
				\|K_{n,s}f\|^p_p &\leq \int_{-1}^{1}\int_{-1}^{1} \left| \displaystyle \sum_{k=-n}^{n}\sum_{m=-n}^{n} n^2\int_{k/n}^{(k+1)/n} \int_{m/n}^{(m+1)/n} f(u,v) L_{k,m,n}(\mathscr{z}) dvdu\right|^p dydx \nonumber\\
				&\leq \int_{-1}^{1}\int_{-1}^{1} \displaystyle \sum_{k=-n}^{n}\sum_{m=-n}^{n}  L_{k,m,n}(\mathscr{z}) \left|n^2 \int_{k/n}^{(k+1)/n} \int_{m/n}^{(m+1)/n} f(u,v)dvdu \right|^p dydx\nonumber\\
					&\leq  n^2 \displaystyle \sum_{k=-n}^{n}\sum_{m=-n}^{n}   \int_{k/n}^{(k+1)/n} \int_{m/n}^{(m+1)/n}  \left|f(u,v)\right|^p dvdu \int_{-1}^{1}\int_{-1}^{1} L_{k,m,n}(\mathscr{z})dydx.
		\end{flalign}
	From Lemma \ref{lem 4.1}, one can observe that
		\begin{flalign}\label{e2}
		\int_{-1}^{1}\int_{-1}^{1}   L_{k,m,n}(\mathscr{z})dydx &=	\displaystyle \sum_{p=-n}^{n}\sum_{q=-n}^{n} \int_{p/n}^{(p+1)/n} \int_{q/n}^{(q+1)/n}L_{k,m,n}(x,y)dydx \nonumber\\
		&\leq   D_s 	\displaystyle \sum_{p=-n}^{n}\sum_{q=-n}^{n}  \int_{p/n}^{(p+1)/n} \int_{q/n}^{(q+1)/n}  \left(|k-p| +1 \right)^{-s}\left(|m-q| +1 \right)^{-s} dydx\nonumber \\
			&\leq    \frac{D_s }{n^2}	\displaystyle \sum_{p=-n}^{n}\sum_{q=-n}^{n}\left(|k-p| +1 \right)^{-s}\left(|m-q| +1 \right)^{-s}\nonumber\\
			&= 		\mathcal{O}\left( \frac{1}{n^2} \right) ,
		\end{flalign}
	holds whenever $s \geq 2.$	Hence from (\ref{e1}) and (\ref{e2}), we get
		\begin{flalign*}
				\|K_{n,s}f\|_p
				& \leq C\|f\|_p.
			\end{flalign*}
		This completes the proof.
	\end{proof}

	In the following theorem we establish the  convergence of (\ref{kant}) in $L^p(X)$. 
		
		\begin{thm}\label{thm 4.6}
			 Let $f \in L^p(X)$  and $s \geq 2$. Then we have
		\[
		\lim_{n \to \infty} \left\|  K_{n,s}f - f \right\|_p = 0.\]
		\end{thm}

	\begin{proof}
			The proof is established by utilizing density argument. Assume that $f\in L^{p}(X)$ and $\epsilon >0$. Since $ C(X)$ is dense in $L^{p}(X)$ (\cite{rudin}), there exists $g \in C(X)$  such that
		 $\|f - g \|_p < \epsilon/2(C+1)$. 
		 Now using triangle inequality, we get
		\[
		\|  K_{n,s}f - f \|_p \leq \|  K_{n,s}f- K_{n,s}g \|_p + \|  K_{n,s}g- g \|_p + \| f - g \|_p.
		\]
		Hence, by Lemma \ref{lem 4.5} and Theorem \ref{thm 4.4}, we obtain
		\begin{flalign}\label{eq}
		\|  K_{n,s}f - f \|_p &\leq (C +1)\| f - g \|_p + \|  K_{n,s}g - g \|_p\\
		& \leq \epsilon/2 +\epsilon/2= \epsilon \nonumber.
		\end{flalign}
		 This proves the desired result.
	\end{proof}

	In the following result, we employ the well-known \emph{Hardy–Littlewood maximal function} \cite{har}, defined as
		\[
	M(f,\mathscr{z}) = \sup_{r>0} \frac{1}{|B(\mathscr{z},r)|} \int_{B(\mathscr{z},r)} |f(w)| \, dA(w),
	\]
	for locally integrable function \( f : X \to \mathbb{C} \). The celebrated theorem of Hardy, Littlewood and Wiener asserts that \( \left(	M(f,\mathscr{z})  \right)\) is bounded on \( L^p(X) \) for \( 1 < p \leq +\infty \), i.e.,
	\begin{equation}\label{max}
		\|Mf\|_{p} \le C_p \|f\|_{p}, 
	\end{equation}
	where \( C_p \) is constant depending only on \( p \).

	\begin{thm}
		Let   $f\in L^p(X), \, p> 1$. Then we have
		\[
		\|  K_{n,s}f-f \|_{p} \leq C_{p,s} \hspace{1pt}\omega(f, \varepsilon_n)_p,
		\]
		 where
		\begin{equation}\label{eps}
		\varepsilon_n =
		\begin{cases}
			n^{-1}, & \text{if } s > 2, \\
			n^{-1} \log n, & \text{if } s = 2.
		\end{cases}
		\end{equation}
		\end{thm}
		\begin{proof}
			Consider $g\in AC(X)$ and $g'\in L^p(X)$. Using (\ref{eq}), we have
		\begin{flalign}\label{4.8}
			\|K_{n,s}f-f\|_{p} 
			&\le (C+1)\|f-g\|_{p} + \|K_{n,s}g-g\|_{p}. 
		\end{flalign}
			Utilizing (\ref{kant}),  we obtain
			\begin{flalign*}
			|K_{n,s}g-g| &\le n^2\sum_{k=-n}^n \sum_{m=-n}^n \int_{k/n}^{(k+1)/n}\int_{m/n}^{(m+1)/n} |g(t_1,t_2)-g(x_1,x_2)|\,L_{k,m,n}(\mathscr{z})\,dt_1dt_2\\
		& \le n^2 M(D^{\alpha} g,\mathscr{z}) \sum_{k=-n}^n \sum_{m=-n}^n L_{k,m,n}(\mathscr{z}) \int_{k/n}^{(k+1)/n}\int_{m/n}^{(m+1)/n} \sqrt{(t_1-x_1)^2+(t_2-x_2)^2} \ dt_1dt_2\\
			&\le  M(D^{\alpha} g,\mathscr{z}) K_{n,s}(\psi,\mathscr{z}).
			\end{flalign*}
			In view of  Lemma \ref{lem 4.2} and (\ref{max}), we can write
			\[
			\|K_{n,s}g-g\|_{p} \le C_p\varepsilon_n \|D^{\alpha} g\|_{p}, \quad C_p>0,
			\]
			where $\varepsilon_n$ is as given in (\ref{eps}). Now using (\ref{o1}) and (\ref{4.8}), we obtain
			\begin{flalign*}
			\|K_{n,s}f-f\|_{p} &\le  (C+1)\|f-g\|_{p} + C_p\varepsilon_n \|D^{\alpha} g\|_{p}\\
			& \leq  C_{p,s} \hspace{1pt}\omega(f, \varepsilon_n)_p,
			\end{flalign*}
			where $C_{p,s}=max \left\lbrace C+1,C_p \right\rbrace $. This proved the desired result.
			\end{proof}

		\section{Approximation by complex Hermite type Nevai Operators}\label{5}
		For any complex-valued  $r$-times differentiable function $f : X \rightarrow \CC$, the complex  Hermite type Nevai operator  for $r,n \in \NN$ is defined as 
		\begin{flalign}\label{herm}
		H^{(r)}_{n,s}(f, \mathscr{z}) 
		&= \frac{\displaystyle \sum_{k=1}^{n}\sum_{m=1}^{n} \lambda_{k,m} \left| K_n(\mathscr{z}, \mathscr{z}_{k,m}) \right|^s \left( \sum_{j=0}^r \frac{f^{(j)}(\mathscr{z}_{k,m})}{j!} (\mathscr{z} - \mathscr{z}_{k,m})^j \right)}{ \displaystyle\sum_{k=1}^{n}\sum_{m=1}^{n} \lambda_{k,m} \left| K_n(\mathscr{z}, \mathscr{z}_{k,m}) \right|^s  }, \quad s \in [0,\infty).
	\end{flalign}

In order to prove the convergence result for operator (\ref{herm}), we first establish the following result.

		\begin{lemma}\label{lem 5.1}
			Let \( r \in \mathbb{N} \). Then, for every \( f \in C^r(X) \), the following inequality holds:
			\[
				\left| 	f(\mathscr{z}) - \sum_{j=0}^{r} \frac{f^{(j)}(\mathscr{z}_{k,m})}{j!} (\mathscr{z} - \mathscr{z}_{k,m})^j  \right|  \leq   \frac{(\mathscr{z} - \mathscr{z}_{k,m})^r}{r!}   \omega\big( f^{(r)}, |\mathscr{z}-\mathscr{z}_{k,m}|_2 \big) .
			\]
		\end{lemma}

		\begin{proof} 
			Since \( f \in C^r(X) \), from the complex version of Taylor's theorem, we deduce that
		\begin{equation}\label{eq:taylor}
			f(\mathscr{z}) = \sum_{j=0}^{r-1} \frac{f^{(j)}(\mathscr{z}_{k,m})}{j!} (\mathscr{z} - \mathscr{z}_{k,m})^j  + R_{r}(\mathscr{z}),
		\end{equation}
		where the remainder \( R_{r} \) is given  by
		\begin{equation}\label{eq:remainder}
			R_{r}(\mathscr{z}) =  \frac{ (\mathscr{z} - \mathscr{z}_{k,m})^r}{(r-1)!} \int_0^1 (1 - t)^{r-1} f^{(r)}(\mathscr{z}_{k,m}+t(\mathscr{z}-\mathscr{z}_{k,m})) dt.
		\end{equation}
	Using (\ref{eq:remainder}), we can write
		\begin{equation}\label{rr}
			R_{r}(\mathscr{z}) = \frac{ (\mathscr{z} - \mathscr{z}_{k,m})^r}{r!}  f^{(r)}(\mathscr{z}_{k,m}) +  \frac{ (\mathscr{z} - \mathscr{z}_{k,m})^r}{(r-1)!} \int_0^1 (1 - t)^{r-1} f^{(r)}(\mathscr{z}_{k,m}+t(\mathscr{z}-\mathscr{z}_{k,m})-f^{(r)}(\mathscr{z}_{k,m})) dt.
		\end{equation}
	In view of \eqref{eq:taylor}-\eqref{rr}, we obtain
		\begin{flalign*}
			\left| 	f(\mathscr{z}) - \sum_{j=0}^{r} \frac{f^{(j)}(\mathscr{z}_{k,m})}{j!} (\mathscr{z} - \mathscr{z}_{k,m})^j \right| \leq   \frac{(\mathscr{z} - \mathscr{z}_{k,m})^r}{(r-1)!} \int_0^1 (1 - t)^{r-1} 
			\left| f^{(r)}(\mathscr{z}_{k,m}+t(\mathscr{z}-\mathscr{z}_{k,m}))-f^{(r)}(\mathscr{z}_{k,m}) \right|dt.
		\end{flalign*}
		By using the fundamental propertie of the modulus of continuity, we obtain 
		\begin{align*}
			\left| 	f(\mathscr{z}) - \sum_{j=0}^{r} \frac{f^{(j)}(\mathscr{z}_{k,m})}{j!} (\mathscr{z} - \mathscr{z}_{k,m})^j  \right| &\leq   \frac{(\mathscr{z} - \mathscr{z}_{k,m})^r}{(r-1)!}  \int_0^1 (1 - t)^{r-1} \omega\big( f^{(r)}, |\mathscr{z}-\mathscr{z}_{k,m}|_2 \big) dt .
		\end{align*}
		This gives
		$$	\left| 	f(\mathscr{z}) - \sum_{j=0}^{r} \frac{f^{(j)}(\mathscr{z}_{k,m})}{j!} (\mathscr{z} - \mathscr{z}_{k,m})^j  \right| \leq   \frac{(\mathscr{z} - \mathscr{z}_{k,m})^r}{r!}   \omega\big( f^{(r)}, |\mathscr{z}-\mathscr{z}_{k,m}|_2 \big) .$$
	\end{proof}

	\begin{thm}\label{thm 5.2}
			Let $f \in C^r(X)$, $1 < s \leq 2$.
		Then we have
		\begin{equation*}
			\left| H^{(r)}_{n,s}(f,\mathscr{z}) - f(\mathscr{z}) \right|
		 \leq 
		c'\omega\left( f^{(r)},  \displaystyle \frac{\left(B_2|T_n(x)|^s+ B_1|T_n(y)|^s \right)}{\lambda_n^{1-s}(\mathscr{z})\left[ 1 + \ln n \right]}\right), 
		\end{equation*}
		where  $c'=\left( \frac{(\mathscr{z} - \mathscr{z}_{k,m})^r}{r!} +1\right) $.
	\end{thm}
%
	
	\begin{proof}
	In view of (\ref{ome}), (\ref{herm}) and Lemma \ref{lem 5.1}, we have
		\begin{flalign*}
	\left| H^{(r)}_{n,s}(f,\mathscr{z}) - f(\mathscr{z}) \right|
	&\leq
	\frac{\displaystyle \sum_{k=1}^{n} \sum_{m=1}^{n} \left|\frac{f^{(j)}(\mathscr{z}_{k,m})}{j!} (\mathscr{z} - \mathscr{z}_{k,m})^j - f(\mathscr{z})\right|
		\, \lambda_{k,m} |K_n(\mathscr{z},\mathscr{z}_{k,m})|^s}
	{\displaystyle \sum_{k=1}^{n} \sum_{m=1}^{n} \lambda_{k,m} |K_n(\mathscr{z},\mathscr{z}_{k,m})|^s}\\
		&\leq
	\frac{\displaystyle \sum_{k=1}^{n} \sum_{m=1}^{n}  \frac{(\mathscr{z} - \mathscr{z}_{k,m})^r}{r!} \omega\big( f^{(r)},  |\mathscr{z}-\mathscr{z}_{k,m}|_2 \big)
		\, \lambda_{k,m} |K_n(\mathscr{z},\mathscr{z}_{k,m})|^s}
	{\displaystyle \sum_{k=1}^{n} \sum_{m=1}^{n} \lambda_{k,m} |K_n(\mathscr{z},\mathscr{z}_{k,m})|^s}\\
		&\leq
	\frac{\displaystyle \sum_{k=1}^{n} \sum_{m=1}^{n}  \frac{(\mathscr{z} - \mathscr{z}_{k,m})^r}{r!}  \left( 1+\frac{|\mathscr{z}-\mathscr{z}_{k,m}|_2}{\delta}\right)\omega\big( f^{(r)},  \delta \big)
		\, \lambda_{k,m} |K_n(\mathscr{z},\mathscr{z}_{k,m})|^s}
	{\displaystyle \sum_{k=1}^{n} \sum_{m=1}^{n} \lambda_{k,m} |K_n(\mathscr{z},\mathscr{z}_{k,m})|^s}.
	\end{flalign*}
	Using Lemma \ref{lem 3.3} and \ref{lem 3.7}, we obtain
	\begin{flalign*}
	\left| H^{(r)}_{n,s}(f,\mathscr{z}) - f(\mathscr{z}) \right|	&\leq \omega\big( f^{(r)},  \delta \big) \left\lbrace \displaystyle \frac{(\mathscr{z} - \mathscr{z}_{k,m})^r}{r!} + \frac{\lambda_n(\mathscr{z})^{s-1}}{\delta}\displaystyle \sum_{k=1}^{n} \sum_{m=1}^{n}|\mathscr{z}-\mathscr{z}_{k,m}|_2  \lambda_{k,m} |K_n(\mathscr{z},\mathscr{z}_{k,m})|^s \right\rbrace \\
	&\leq \omega\big( f^{(r)},  \delta \big) \left\lbrace \displaystyle \frac{(\mathscr{z} - \mathscr{z}_{k,m})^r}{r!} + \frac{\lambda_n^{1-s}(\mathscr{z})\left[ 1 + \ln n \right]}{\delta} \left(B_2|T_n(x)|^s+ B_1|T_n(y)|^s \right) \right\rbrace .
	\end{flalign*}
	Now choosing
	$$\delta= \displaystyle \frac{\left(B_2|T_n(x)|^s+ B_1|T_n(y)|^s \right)}{\lambda_n^{1-s}(\mathscr{z})\left[ 1 + \ln n \right]},$$
	  yields the desired result.
	\end{proof}

	\section{Numerical Simulations and Application}\label{6}
	
	\subsection{Approximation by complex generalized Nevai operators}\label{6.1}
To validate Theorem \ref{thm 3.8}, we present Figure \ref{g1} and Table \ref{t1}, illustrating the approximation  of $f_1$ by complex generalized Nevai operators (\ref{Nevai}).
Here we consider example of a non-analytic function defined as
 $$ f_1(\mathscr{z}) = e^{-|\mathscr{z}|^2} (\mathscr{z} - \overline{\mathscr{z}}) e^{i 6 \arg(\mathscr{z})}, \quad  \mathscr{z} \in X.$$ 

It is evident from  Figure \ref{g1} and Table \ref{t1} that the approximation gets better as we increase the parameter $n.$
 In order to  address different aspects of approximation, we compute the following errors: 
 	\begin{itemize}
 	\item The maximum error by $ e_{\max}^{\Re, \Im} :=\displaystyle \max_{1 \le j \le n_e} e_j^{\Re, \Im}$.
 	\item The  mean error by	$e_{\text{mean}}^{\Re, \Im} := \frac{1}{n_e} \displaystyle\sum_{j=1}^{n_e} e_j^{\Re, \Im}.$
 	\item  The  mean squared error  by $	e_{MS}^{\Re, \Im} := \sqrt{
 		\frac{1}{n_e} \displaystyle \sum_{j=1}^{n_e} \left(e_j^{\Re, \Im}\right)^2
 	}.$
 \end{itemize}

\begin{figure}[h!]
	\centering
	
	\begin{subfigure}[c]{0.32\textwidth}
		\centering
		\includegraphics[width=\linewidth]{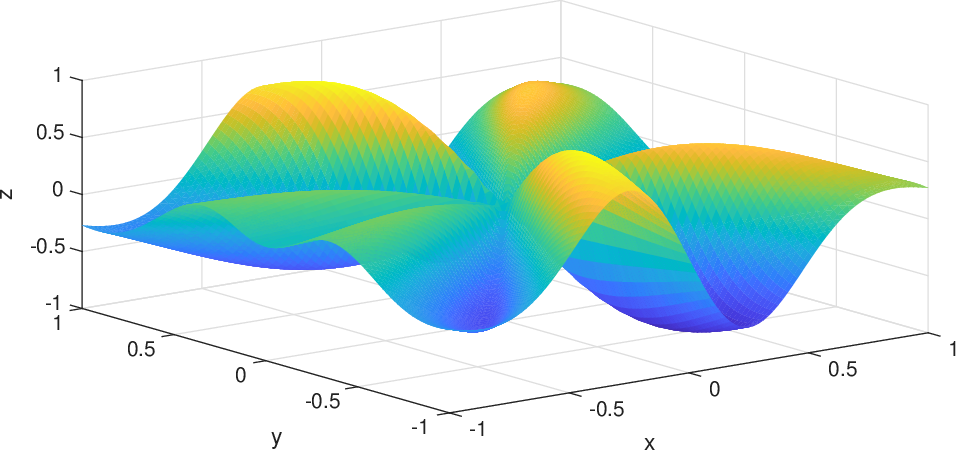}
		\caption{Real part of $f_1$}
	\end{subfigure}
	\hfill
	\begin{subfigure}[c]{0.32\textwidth}
		\centering
		\includegraphics[width=\linewidth]{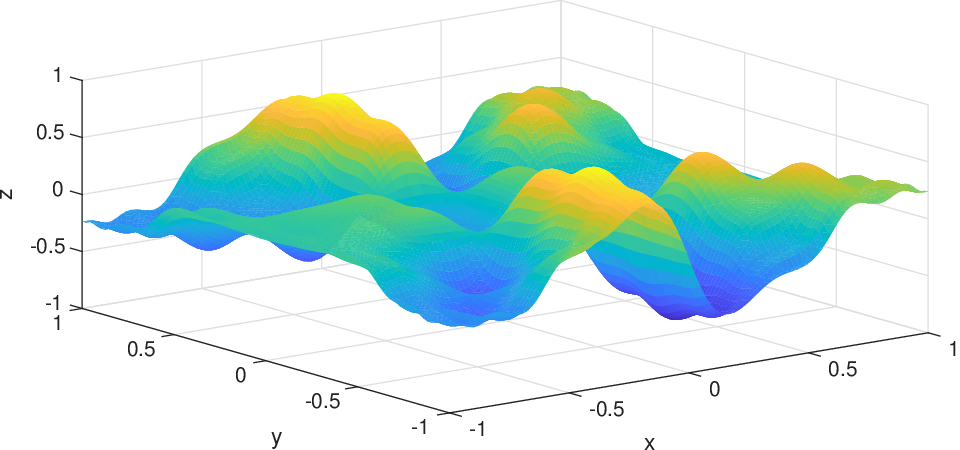}
			\caption{Real part of $f_1$  by $(N_{10,2}f_1)$}
	\end{subfigure}
	\hfill
	\begin{subfigure}[c]{0.32\textwidth}
		\centering
		\includegraphics[width=\linewidth]{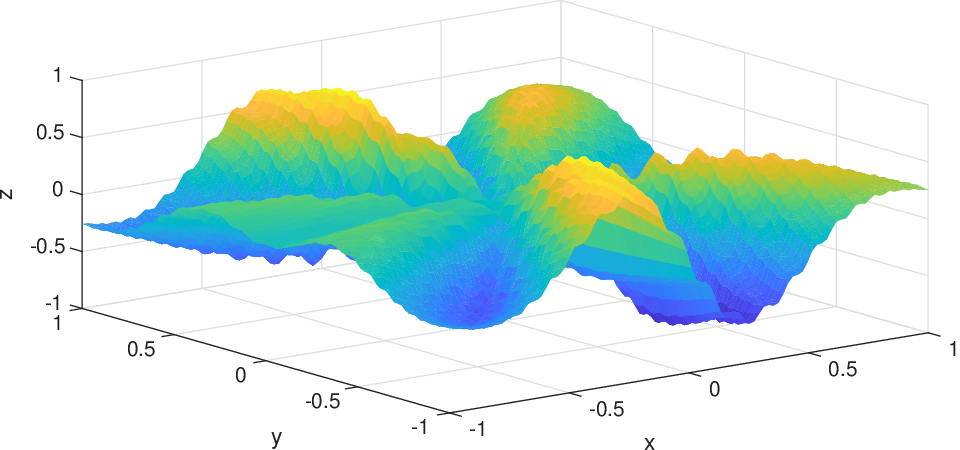}
		\caption{Real part of $f_1$  by $(N_{30,2}f_1)$}
	\end{subfigure}
	
	\vspace{3mm}
	
	\begin{subfigure}[c]{0.32\textwidth}
		\centering
			\includegraphics[width=\linewidth]{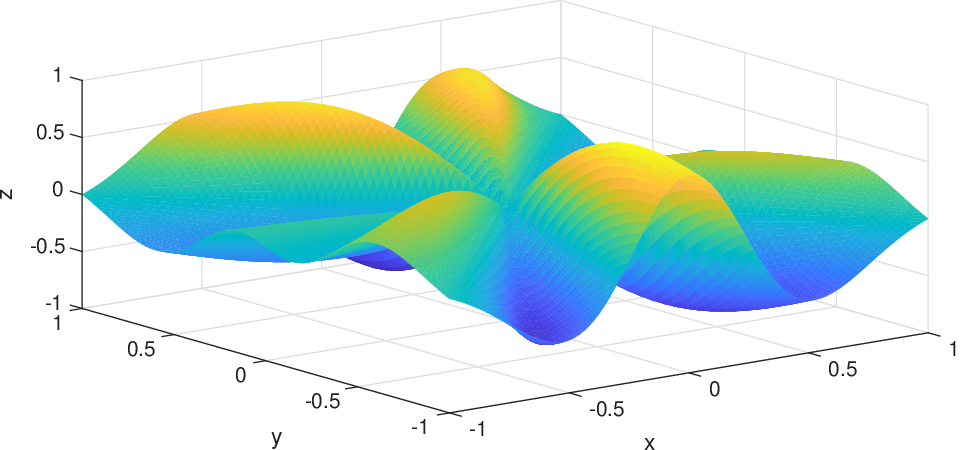}
					\caption{Imaginary part of $f_1$}
	\end{subfigure}
	\hfill
	\begin{subfigure}[c]{0.32\textwidth}
		\centering
		\includegraphics[width=\linewidth]{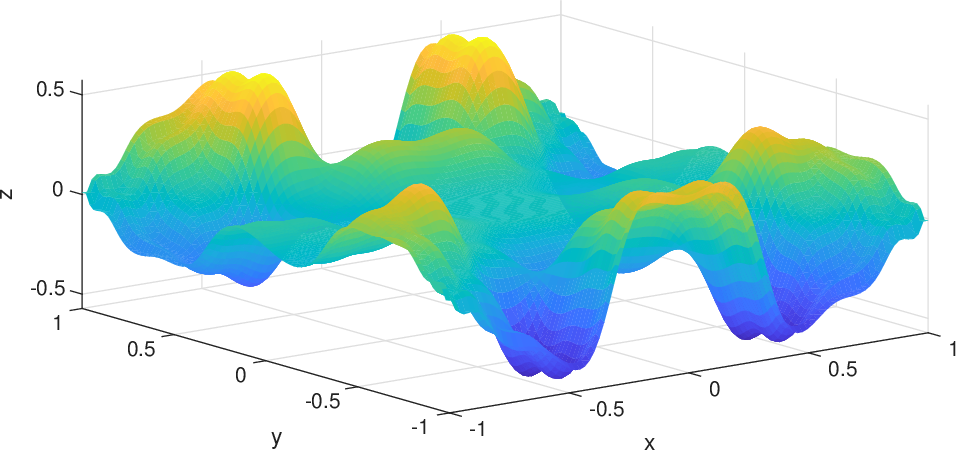}
		\caption{Imaginary part of $f_1$ by $(N_{10,2}f_1)$}
	\end{subfigure}
	\hfill
	\begin{subfigure}[c]{0.32\textwidth}
		\centering
		\includegraphics[width=\linewidth]{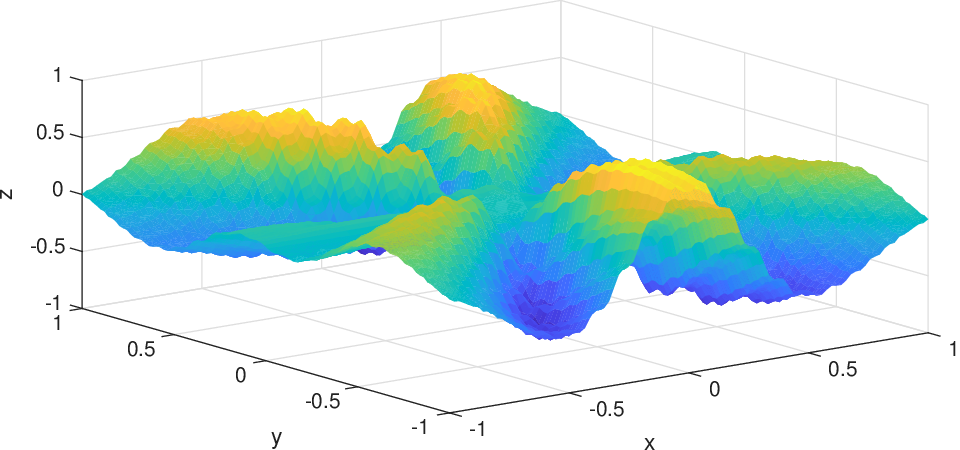}
		\caption{Imaginary part of $f_1$ by $(N_{30,2}f_1)$}
	\end{subfigure}
	
	\caption{Approximation of the real and imaginary  parts of  $f_1$ by $(N_{n,s}f_1)$ for $s=2$ and  different values of $n$ }
	\label{g1}
\end{figure}

\begin{table}[ht]
	\centering
	\caption{Errors in approximation of $f_1$ by  $(N_{n,2}f_1)$ for different values of $n$}
	\label{t1}  \renewcommand{\arraystretch}{1.5}
		{	\rowcolors{2}{blue!10}{white} 
			\begin{tabular}{>{\bfseries}c c c c c c c}
				\rowcolor{headerblue}\color{white}
				n & $e_{max}^\Re$ & $e_{mean}^\Re$ & $e_{MS}^\Re$ & $e_{max}^\Im$ & $e_{mean}^\Im$ & $e_{MS}^\Im$ \\
				\hline
				10 & 6.076e-01 & 1.291e-01 & 1.672e-01 & 9.642e-01 & 1.725e-01 & 2.489e-01 \\
				20 & 3.939e-01 & 8.309e-02 & 1.113e-01 & 5.606e-01 & 8.968e-02 & 1.287e-01 \\
				30 & 3.263e-01 & 6.054e-02 & 8.196e-02 & 3.782e-01 & 6.299e-02 & 8.861e-02 \\
				40 & 2.518e-01 & 4.741e-02 & 6.430e-02 & 2.773e-01 & 4.848e-02 & 6.790e-02 \\
				50 & 2.066e-01 & 3.818e-02 & 5.281e-02 & 2.129e-01 & 3.947e-02 & 5.521e-02 \\
				\hline
			\end{tabular}
		}
	\end{table}

		\subsubsection{Approximation of  contour lines by complex generalized Nevai operators}
		Now, we demonstrate how well a complex generalized Nevai operator (\ref{Nevai}) can approximate contour lines of a non-analytic function given by
			$$g_1(\mathscr{z})=	\overline{\mathscr{z}}, \quad \mathscr{z} \in X.$$
			

\begin{table}[ht]
	\centering
	\caption{Errors in approximation of $|g_1|$ by $|N_{n,2}g_1|$ for different $n$}
	\renewcommand{\arraystretch}{1.3}{	\rowcolors{2}{blue!10}{white}
		\begin{tabular}{>{\bfseries}c c c c }
			\hline
			\rowcolor{headerblue}\color{white}
			$n$ & \multicolumn{3}{c}{Error Analysis}  \\
			\cmidrule(lr){2-4}
			& Maximum error & Mean  error & Mean squared error  \\
			\hline
			10 & 2.238e-01 & 9.819e-02 & 1.064e-01  \\
			20 & 1.070e-01 & 5.408e-02 & 5.836e-02  \\
			30 & 8.492e-02 & 3.767e-02 & 4.073e-02 \\
			40 & 5.903e-02 & 2.776e-02 & 2.973e-02 \\
			50 & 4.515e-02 & 2.244e-02 & 2.426e-02  \\
			\hline
	\end{tabular}}
	\label{g11}
\end{table}

			In Figure \ref{n1}, the original contour lines of $|g_1|$ are shown. Figures \ref{n2}-\ref{n3} present the approximation of contour lines of $|g_1|$ by $|N_{n,2}g_1|$ for $n=10,20$, while Figures \ref{n4}-\ref{n5} illustrate the corresponding absolute error in the approximation for $n=10,20$. As shown in Figure \ref{cont1} and Table \ref{g11},  the operator (\ref{Nevai}) performs better as we   increases the parameter $n$.
	 

		
		\begin{figure}[htbp]
		\centering
		
		\begin{subfigure}[b]{0.25\textwidth}
			\centering
			\includegraphics[width=\textwidth]{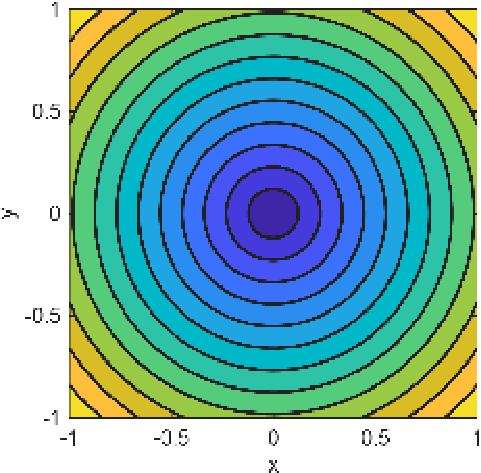}
			\caption{$|g_1|$}
			\label{n1}
		\end{subfigure}
			\hspace{0.05\textwidth} 
		\begin{subfigure}[b]{0.25\textwidth}
			\centering
			\includegraphics[width=\textwidth]{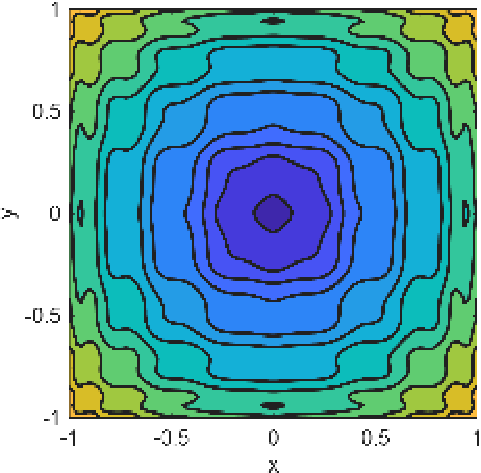}
			\caption{$|N_{10,2}g_1|$}
			\label{n2}
		\end{subfigure}
			\hspace{0.05\textwidth} 
		\begin{subfigure}[b]{0.25\textwidth}
			\centering
			\includegraphics[width=\textwidth]{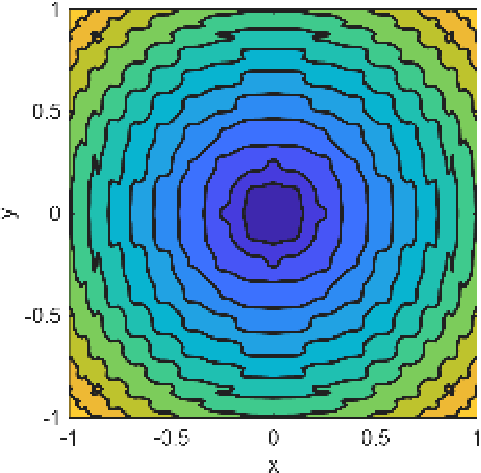}
			\caption{$|N_{20,2}g_1|$}
			\label{n3}
		\end{subfigure}
		
		\vspace{0.5cm} 
		\begin{subfigure}[b]{0.25\textwidth}
			\centering
			\includegraphics[width=\textwidth]{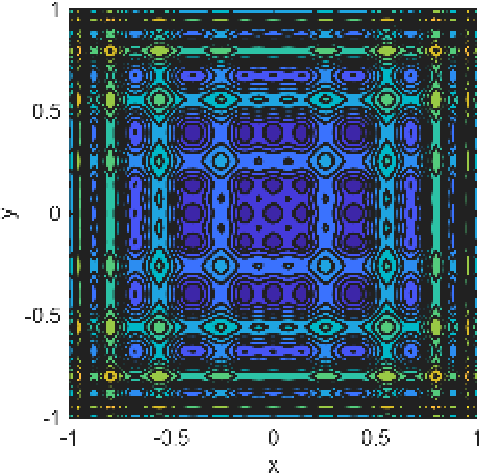}
			\caption{Absolute error for $n=10$}
			\label{n4}
		\end{subfigure}
			\hspace{0.05\textwidth} 
		\begin{subfigure}[b]{0.25\textwidth}
			\centering
			\includegraphics[width=\textwidth]{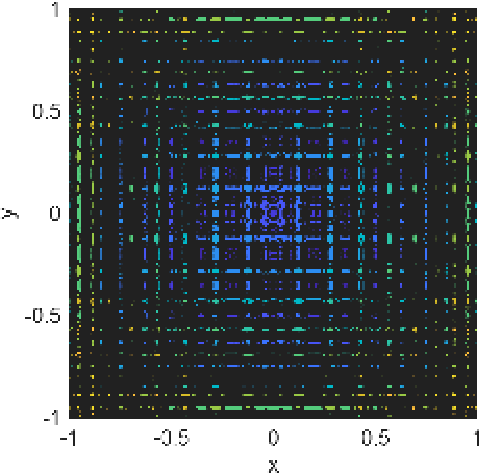} 
			\caption{Absolute error for $n=20$}
			\label{n5}
		\end{subfigure}
		
		\caption{Contour lines of  $|g_1|$, $|N_{n,2}g_1|$ and absolute error for  $n=10,20$}
		\label{cont1}
	\end{figure}

		\subsection{Approximation by complex  Kantorovich type Nevai operators}
	Here, we demonstrate the effectiveness of the complex  Kantorovich type Nevai operators (\ref{kant}) to approximate non analytic but $p-$integrable function defined as
			\[
		f_2(\mathscr{z}) = \left\lfloor 3\,\Re(\mathscr{z}) \right\rfloor 
		+ i\left( 1 - \tfrac{1}{2}\Re(\mathscr{z})^{2} - \tfrac{1}{2}\Im(\mathscr{z})^{2} \right), \quad \mathscr{z} \in X.
		\]
 Figure \ref{g2}  and Table \ref{t2} demonstrate the convergence of (\ref{kant}) for $p-$integrable function  (as established in Theorem \ref{thm 4.6}). 


			\begin{figure}[h!]
				\centering
				
				\begin{subfigure}[c]{0.32\textwidth}
					\centering
				\includegraphics[width=\linewidth]{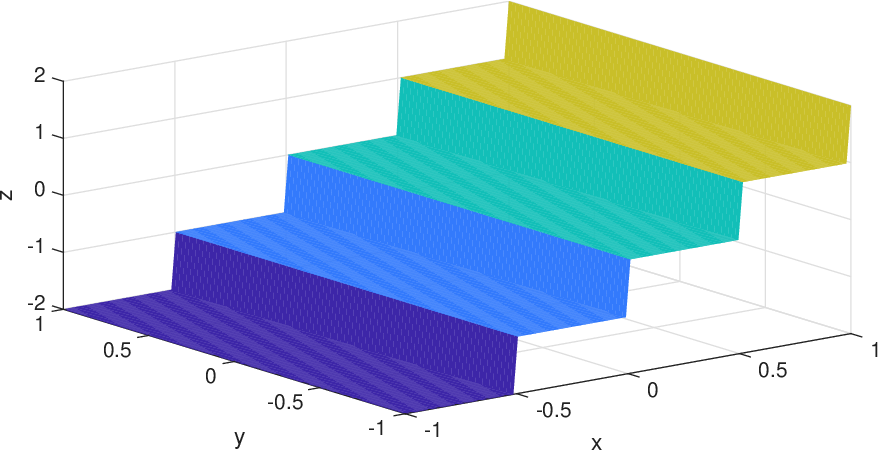}
								\caption{Real part of $f_2$ }
				\end{subfigure}
				\hfill
				\begin{subfigure}[c]{0.32\textwidth}
					\centering
					\includegraphics[width=\linewidth]{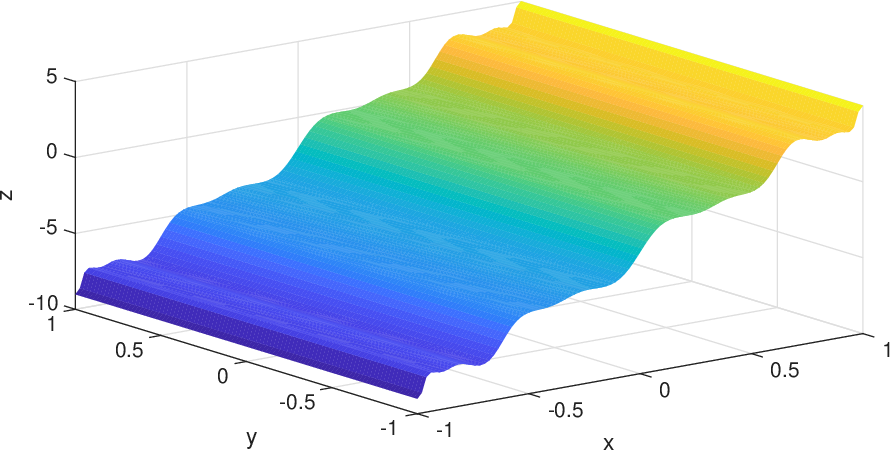}
									\caption{Real part of $f_2$ by $(K_{10,2}f_2)$}
				\end{subfigure}
				\hfill
				\begin{subfigure}[c]{0.32\textwidth}
					\centering
					\includegraphics[width=\linewidth]{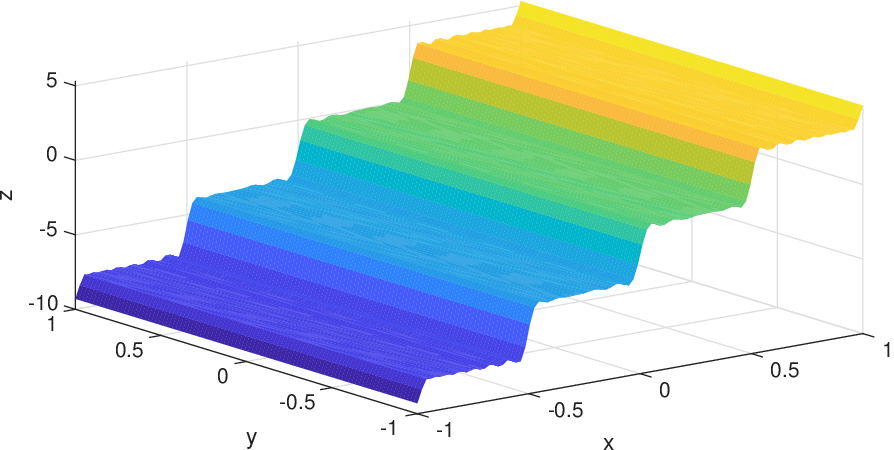}
				\caption{Real part of $f_2$ by $(K_{30,2}f_2)$}
				\end{subfigure}
				
				\vspace{3mm}
				
				\begin{subfigure}[c]{0.32\textwidth}
					\centering
					\includegraphics[width=\linewidth]{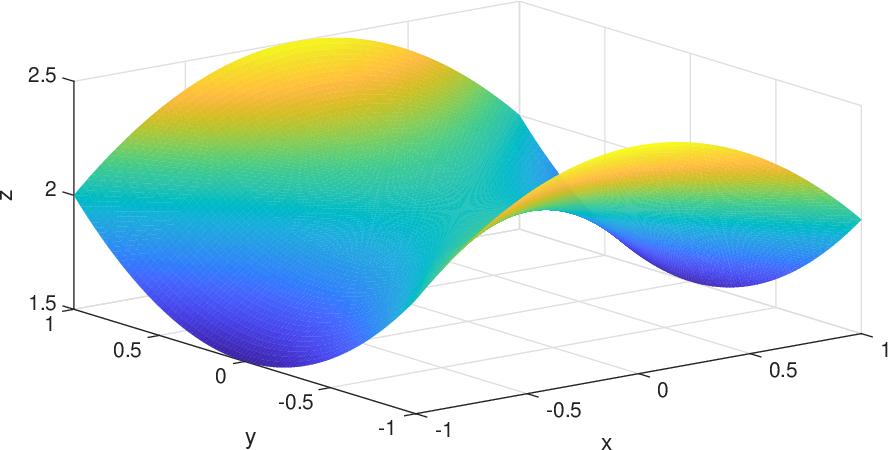}
									\caption{Imaginary part of $f_2$}
				\end{subfigure}
				\hfill
				\begin{subfigure}[c]{0.32\textwidth}
					\centering
			\includegraphics[width=\linewidth]{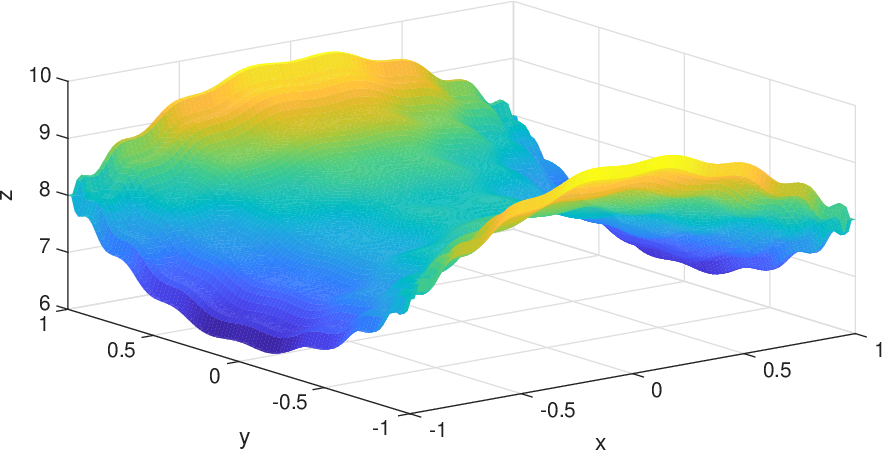}
				\caption{Imaginary part of $f_2$ by $(K_{10,2}f_2)$}
				\end{subfigure}
				\hfill
				\begin{subfigure}[c]{0.32\textwidth}
					\centering
						\includegraphics[width=\linewidth]{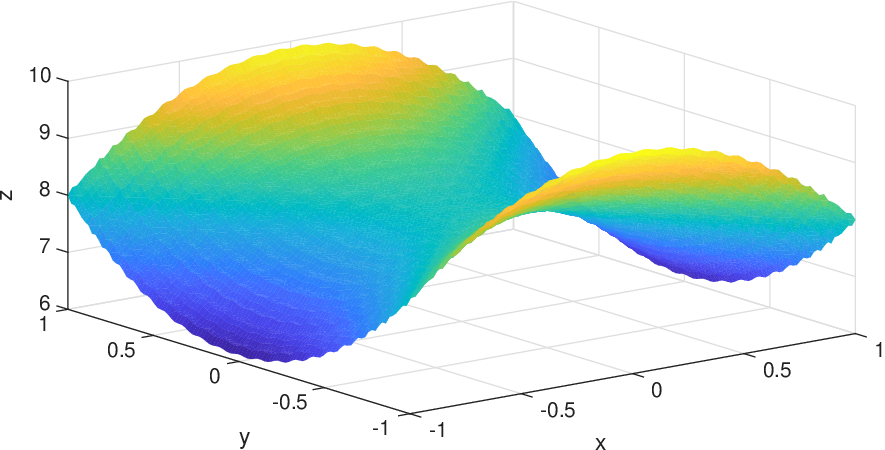}
					\caption{Imaginary part of $f_2$ by $(K_{30,2}f_2)$}
				\end{subfigure}
				
			\caption{Approximation of the real and imaginary  parts of the function $f_2$ by $(K_{n,s}f_2)$ for  $s=2$ and different values of $n$ }
						\label{g2}
			\end{figure}

\begin{table}[ht]
	\centering
	\caption{Errors in approximation  of $f_2$ by $(K_{n,2}f_2)$ for different values of $n$ }
	\label{t2}  \renewcommand{\arraystretch}{1.5}
		{	\rowcolors{2}{blue!10}{white}
			\begin{tabular}{>{\bfseries}c c c c c c c}
				\rowcolor{headerblue}\color{white}
				n & $e_{max}^\Re$ & $e_{mean}^\Re$ & $e_{MS}^\Re$ & $e_{max}^\Im$ & $e_{mean}^\Im$ & $e_{MS}^\Im$ \\
				\hline
				10 & 1.562e+01 & 5.822e+00 & 7.032e+00 &7.473e+00  &3.614e+00  & 4.090e+00 \\
				20 & 1.279e+01 & 4.851e+00 & 5.868e+00 & 7.467e+00 & 3.517e+00 & 4.064e+00 \\
				30 & 1.257e+01 & 4.619e+00 & 5.650e+00 & 7.252e+00 & 3.505e+00 & 3.959e+00 \\
				40 & 1.215e+01 & 4.522e+00 & 5.463e+00 & 7.063e+00 & 3.486e+00 & 3.882e+00 \\
				50 & 1.118e+01 & 4.171e+00 & 5.022e+00 & 6.804e+00 & 3.464e+00 & 3.814e+00 \\
				\hline
		\end{tabular}}
	\end{table}


			\subsubsection{Approximation of contour lines by complex Kantorovich type Nevai operators}
	Here we approximate the contour lines of the $p-$integrable function   given by
	 $$ g_2(\mathscr{z})=
		\begin{cases}
			\mathscr{z}, & \text{if }\operatorname{Re}(\mathscr{z})\geq 0 \\
			2\mathscr{z}, & \text{if }\operatorname{Re}(\mathscr{z})<0.
		\end{cases}$$ 

	\begin{figure}[htbp]
	\centering
	
	\begin{subfigure}[b]{0.25\textwidth}
		\centering
		\includegraphics[width=\textwidth]{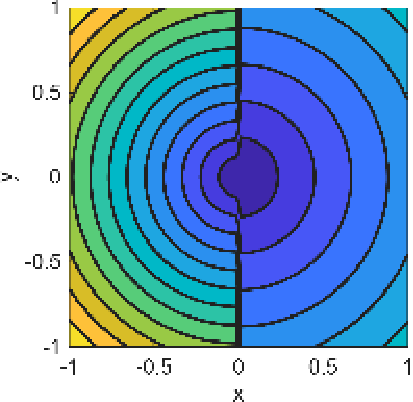}
		\caption{$|g_2|$}
		\label{k1}
	\end{subfigure}
	\hspace{0.05\textwidth} 
	\begin{subfigure}[b]{0.25\textwidth}
		\centering
		\includegraphics[width=\textwidth]{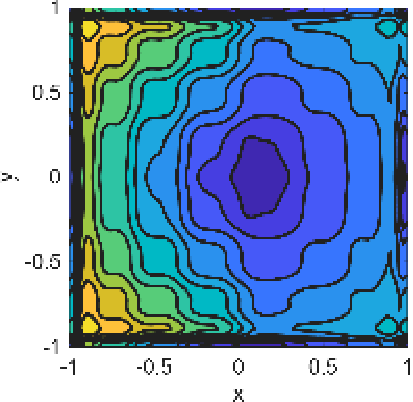}
		\caption{$|K_{10,2}g_2|$}
		\label{k2}
	\end{subfigure}
	\hspace{0.05\textwidth} 
	\begin{subfigure}[b]{0.25\textwidth}
		\centering
		\includegraphics[width=\textwidth]{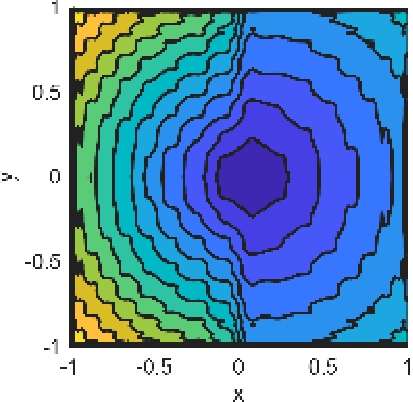}
		\caption{$|K_{20,2}g_2|$}
		\label{k3}
	\end{subfigure}
	
	\vspace{0.5cm} 
	\begin{subfigure}[b]{0.25\textwidth}
		\centering
		\includegraphics[width=\textwidth]{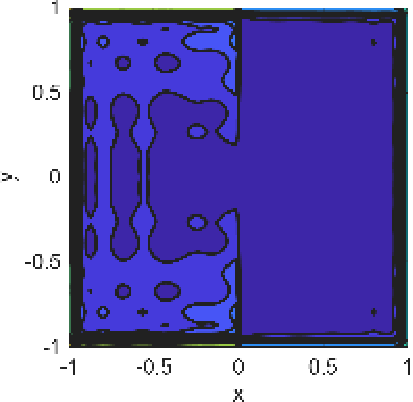}
		\caption{Absolute error for $n=10$}
		\label{k4}
	\end{subfigure}
	\hspace{0.05\textwidth} 
	\begin{subfigure}[b]{0.25\textwidth}
		\centering
		\includegraphics[width=\textwidth]{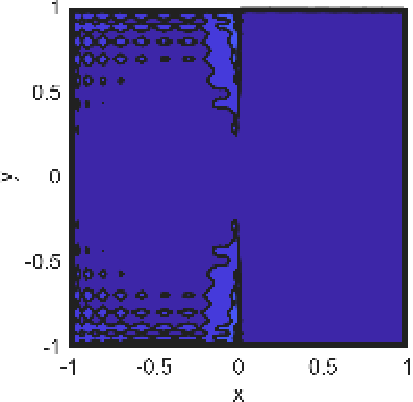} 
		\caption{Absolute error for $n=20$}
		\label{k5}
	\end{subfigure}
	
	\caption{Contour lines of  $|g_2|$, $|K_{n,2}g_2|$  and absolute error for  $n=10,20$}
	\label{cont2}
\end{figure}

	\begin{table}[ht]
	\centering
	\caption{Errors in approximation  of $|g_2|$ by $|K_{n,2}g_2|$ for different $n$ }
	\label{g22}
		\renewcommand{\arraystretch}{1.2}
		{	\rowcolors{2}{blue!10}{white}
			\begin{tabular}{>{\bfseries}c c c c }
				\hline
				\rowcolor{headerblue}\color{white}
				$n$ & \multicolumn{3}{c}{Error Analysis}  \\
				\cmidrule(lr){2-4}
				& Maximum error & Mean  error & Mean squared error \\
				\hline 
				10 & 2.52e+00 & 3.20e-01 & 5.07e-01  \\
				20 & 2.35e+00 & 1.63e-01 & 3.32e-01  \\
				30 & 2.29e+00 & 1.32e-01 & 3.13e-01  \\
				40 & 2.26e+00 & 1.16e-01 & 3.04e-01 \\
				50 & 2.24e+00 & 1.05e-01 & 2.99e-01  \\
				\hline
		\end{tabular}}
	\end{table}

		Figure \ref{k1} shows the original contour lines of $|g_2|$, and approximated contour lines of $|g_2|$ by  $|K_{n,2} g_2|$, for $n=10$ and $n=20$ are presented 
		in Figures \ref{k2}--\ref{k3}. Along with this, the corresponding absolute 
		error plots are given in Figures  \ref{k4}--\ref{k5}. 
		 Hence, we conclude that the complex Kantorovich type Nevai operator (\ref{kant}) provides a better approximation of the contour lines of the $p-$integrable function as we increase the value of $n$ (see 	Figure \ref{cont2} and
		 		Table \ref{g22}).

			\subsection{Approximation by complex  Hermite type Nevai operators}\label{6.3}
		Here, we consider an example of an analytic function  which is defined as 
				$$	f_3(\mathscr{z}) = \cos(\pi \mathscr{z}) + i \sin(\pi \mathscr{z}), \quad \mathscr{z} \in X.$$
	
			
				
				\begin{table}[ht]
					\centering
					\caption{Errors in approximation  of $f_3$ by  $(H^{(3)}_{n,2}f_3)$  for different values of $n$}
					\label{t3}
					\renewcommand{\arraystretch}{1.4}	{	\rowcolors{2}{blue!10}{white}
						\begin{tabular}{>{\bfseries}c c c c c c c}
							\rowcolor{headerblue}\color{white}
							n & $e_{max}^\Re$ & $e_{mean}^\Re$ & $e_{MS}^\Re$ & $e_{max}^\Im$ & $e_{mean}^\Im$ & $e_{MS}^\Im$\\
							\hline
							10 & 8.730e+00 & 1.145e+00 & 1.852e+00 & 1.732e+01 & 2.027e+00 & 3.002e+00 \\
							20 & 4.978e+00 & 7.144e-01 & 1.144e+00 & 9.736e+00 & 1.173e+00 & 1.757e+00 \\
							30 & 3.374e+00 & 5.183e-01 & 8.276e-01 & 6.571e+00 & 8.284e-01 & 1.246e+00 \\
							40 & 2.502e+00 & 4.020e-01 & 6.320e-01 & 4.867e+00 & 6.306e-01 & 9.397e-01 \\
							50 & 1.974e+00 & 3.307e-01 & 5.169e-01 & 3.838e+00 & 5.130e-01 & 7.651e-01 \\
							\hline
					\end{tabular}}
				\end{table}

\begin{figure}[h!]
	\centering
	
	\begin{subfigure}[c]{0.32\textwidth}
		\centering
		\includegraphics[width=\linewidth]{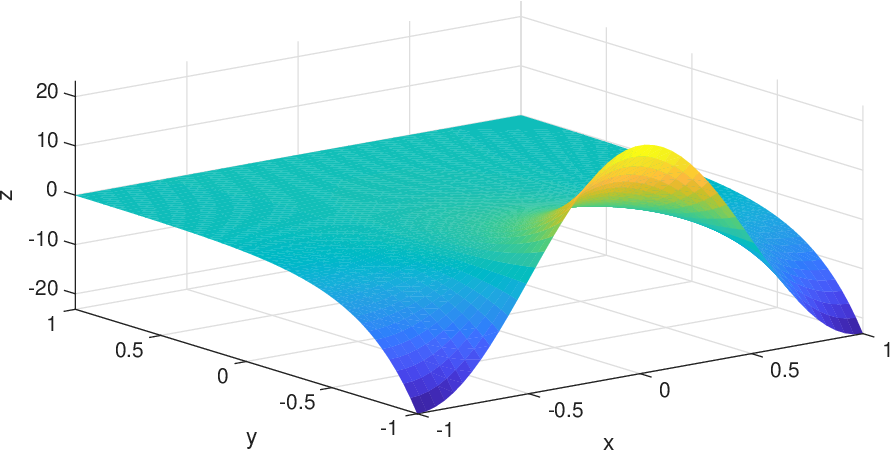}
						\caption{Real part of $f_3$ }
	\end{subfigure}
	\hfill
	\begin{subfigure}[c]{0.32\textwidth}
		\centering
	\includegraphics[width=\linewidth]{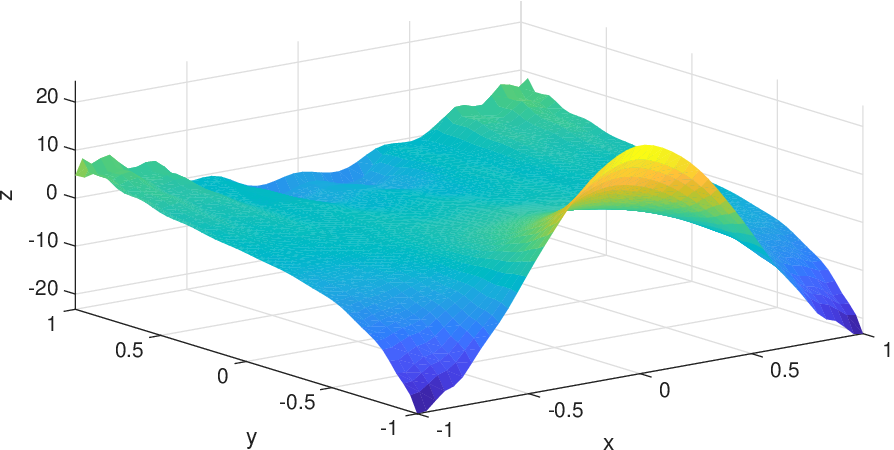}
						\caption{Real part of $f_3$ by $(H^{(3)}_{10,2}f_3)$}
	\end{subfigure}
	\hfill
	\begin{subfigure}[c]{0.32\textwidth}
		\centering
		\includegraphics[width=\linewidth]{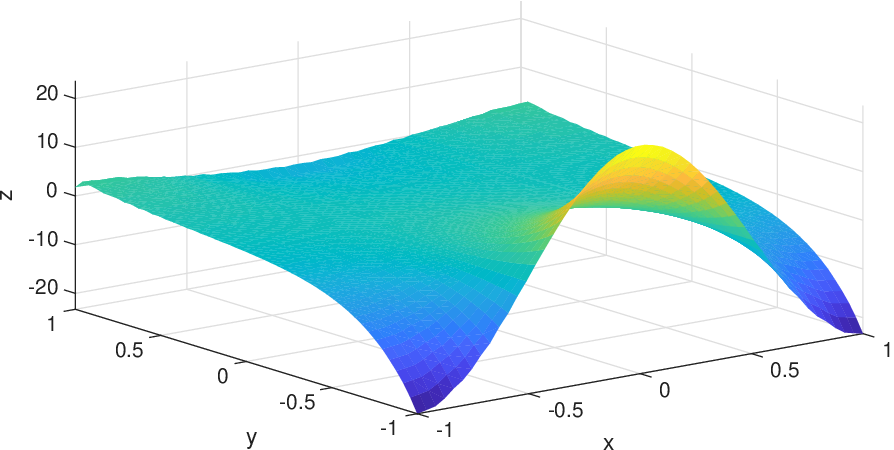}
	\caption{Real part of $f_3$ by $(H^{(3)}_{30,2}f_3)$}
	\end{subfigure}
	
	\vspace{3mm}
	
	\begin{subfigure}[c]{0.32\textwidth}
		\centering
		\includegraphics[width=\linewidth]{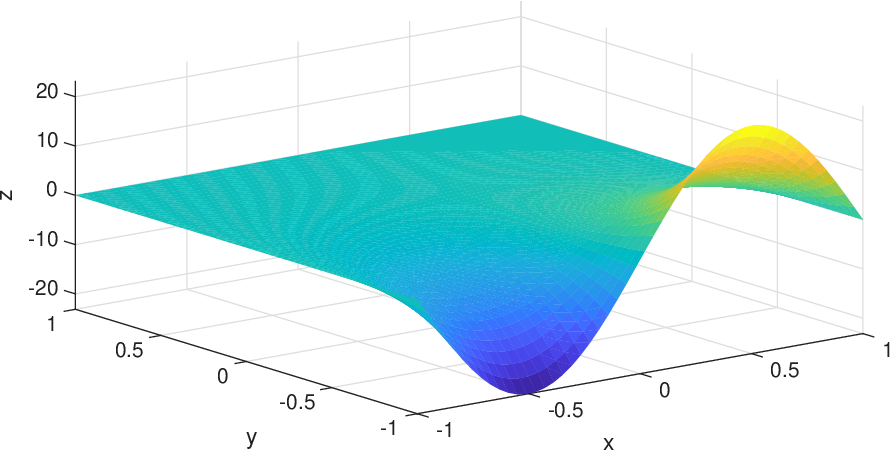}
						\caption{Imaginary part of $f_3$}
	\end{subfigure}
	\hfill
	\begin{subfigure}[c]{0.32\textwidth}
		\centering
		\includegraphics[width=\linewidth]{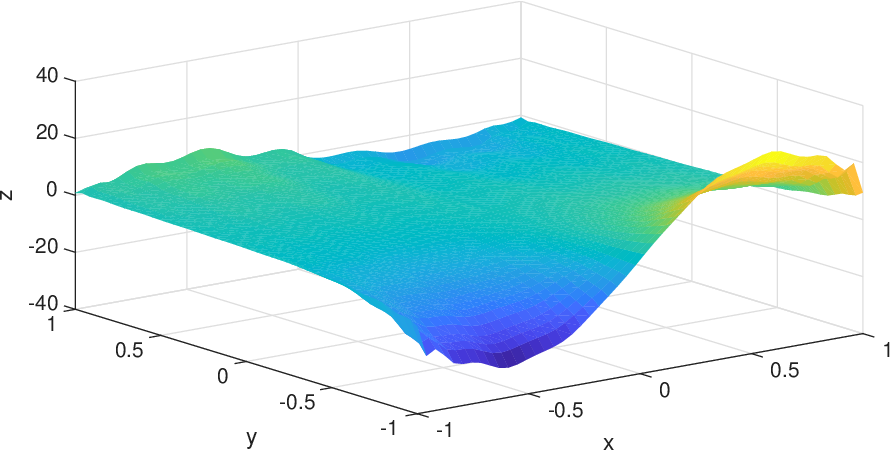}
						\caption{Imaginary part of $f_3$ by $(H^{(3)}_{10,2}f_3)$}
	\end{subfigure}
	\hfill
	\begin{subfigure}[c]{0.32\textwidth}
		\centering
		\includegraphics[width=\linewidth]{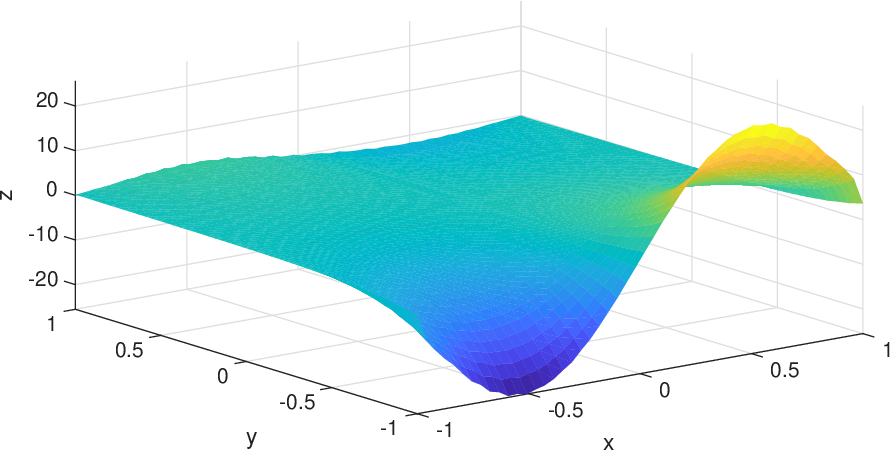}
	\caption{Imaginary part of $f_3$ by $(H^{(3)}_{30,2}f_3)$}
	\end{subfigure}
	
	\caption{Approximation of the real and imaginary  parts of the function $f_3$ by $(H^{(3)}_{n,s}f_3)$ for $s=2$ and different values of $n$  }
				\label{g3}
\end{figure}

				To illustrate the convergence (as established in Theorem \ref{thm 5.2}),  we present Figure \ref{g3}  and Table \ref{t3}. The results show that the performance of the operator (\ref{herm})  improves as $n$ increases.

				\subsubsection{Approximation of  contour by complex Hermite type Nevai operators}
					The contour lines of the original function $|g_3|$ are shown in Figure \ref{h1}, where $g_3$ is given by
					  $$g_3(\mathscr{z})=\sin\left(\mathscr{z}^2\right), \quad  \mathscr{z} \in X.$$

						\begin{figure}[htbp]
					\centering
					
					\begin{subfigure}[b]{0.25\textwidth}
						\centering
						\includegraphics[width=\textwidth]{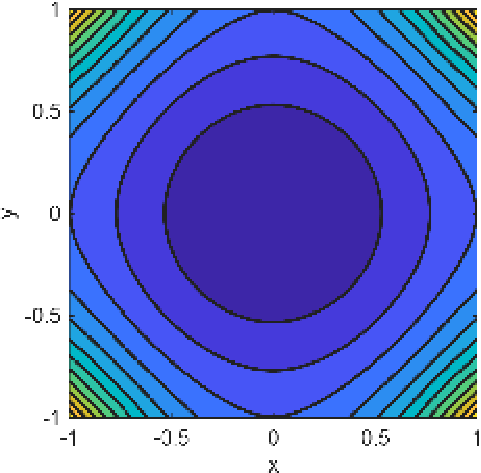}
						\caption{$|g_3|$}
						\label{h1}
					\end{subfigure}
						\hspace{0.05\textwidth} 
					\begin{subfigure}[b]{0.25\textwidth}
						\centering
						\includegraphics[width=\textwidth]{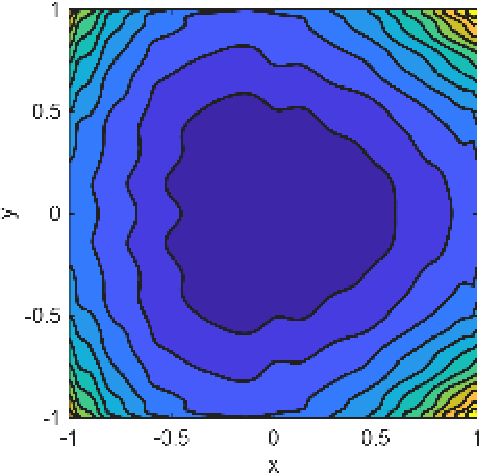}
						\caption{$|H^{(2)}_{10,2}g_3|$}
						\label{h2}
					\end{subfigure}
						\hspace{0.05\textwidth} 
					\begin{subfigure}[b]{0.25\textwidth}
						\centering
						\includegraphics[width=\textwidth]{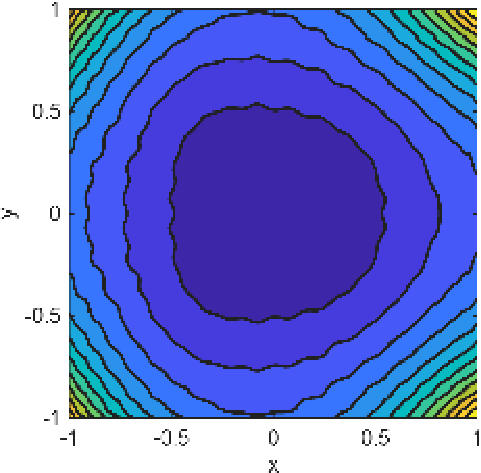}
						\caption{$|H^{(2)}_{20,2}g_3|$}
						\label{h3}
					\end{subfigure}
					
					\vspace{0.5cm} 
					\begin{subfigure}[b]{0.25\textwidth}
						\centering
						\includegraphics[width=\textwidth]{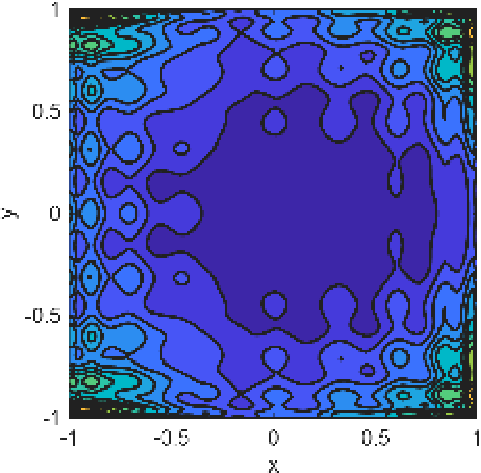}
						\caption{Absolute error for $n=10$}
						\label{h4}
					\end{subfigure}
						\hspace{0.05\textwidth} 
					\begin{subfigure}[b]{0.25\textwidth}
						\centering
						\includegraphics[width=\textwidth]{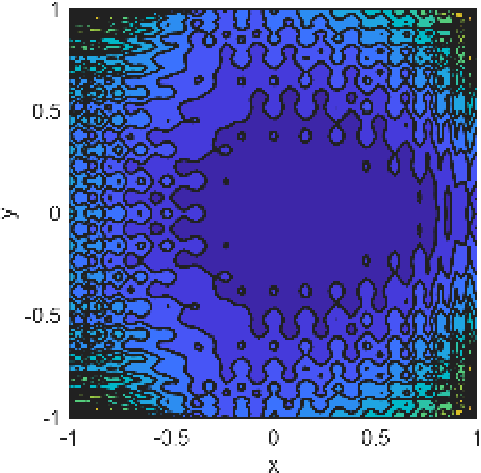} 
						\caption{Absolute error for $n=20$}
						\label{h5}
					\end{subfigure}
					
					\caption{Contour lines of  $|g_3|$, $|H^{(2)}_{n,2}g_3|$ and absolute error for   $n=10,20$}
					\label{cont3}
				\end{figure}
				
					Figures \ref{h2}--\ref{h3} show the contour lines of  $|g_3|$ by
				$|H^{(2)}_{n,2} g_3|$ for $n=10$ and $n=20$,  while Figures
				\ref{h4}--\ref{h5} present  the corresponding absolute error plots. 
				From Figure \ref{cont3} and Table \ref{g33}, we can observe that the complex Hermite type Nevai operator (\ref{herm}) performs well in approximating the contour lines of the analytic function. 
				
				\begin{table}[ht]
					\centering
					\caption{Errors in approximation of $|g_3|$ by $|H^{(2)}_{n,2}g_3|$ for different $n$ }
					\label{g33}
					\renewcommand{\arraystretch}{1.3}	{	\rowcolors{2}{blue!10}{white}
						\begin{tabular}{>{\bfseries}c c c c}
							\hline
							\rowcolor{headerblue}\color{white}
							$n$ & \multicolumn{3}{c}{Error Analysis}  \\
							\cmidrule(lr){2-4}
							& Maximum error & Mean  error & Mean squared error  \\
							\hline
							10 & 7.67e-01 & 1.55e-01 & 2.00e-01  \\
							20 & 4.13e-01 & 9.04e-02 & 1.16e-01  \\
							30 & 2.96e-01 & 6.43e-02 & 8.30e-02  \\
							40 & 2.26e-01 & 4.96e-02 & 6.36e-02 \\
							50 & 1.80e-01 & 4.03e-02 & 5.16e-02 \\
							\hline
					\end{tabular}}
				\end{table}

					
				

			\subsection{Comparison between complex generalized Nevai and complex Hermite type Nevai operators}\label{6.4}
			Now we compare the approximation performance of complex  generalized Nevai operators (\ref{Nevai}) and complex Hermite type Nevai operators (\ref{herm}) in approximating an analytic function. For this,
			 we consider the analytic function  as 
			 $$f_4(\mathscr{z}) = 0.5 \bigl( \mathscr{z}^{4} - 1.5 \mathscr{z}^2 + 0.3 \bigr), \quad  \mathscr{z} \in X.$$

			\begin{figure}[h!]
				\centering
				
				\begin{subfigure}[c]{0.32\textwidth}
					\centering
				\includegraphics[width=\linewidth]{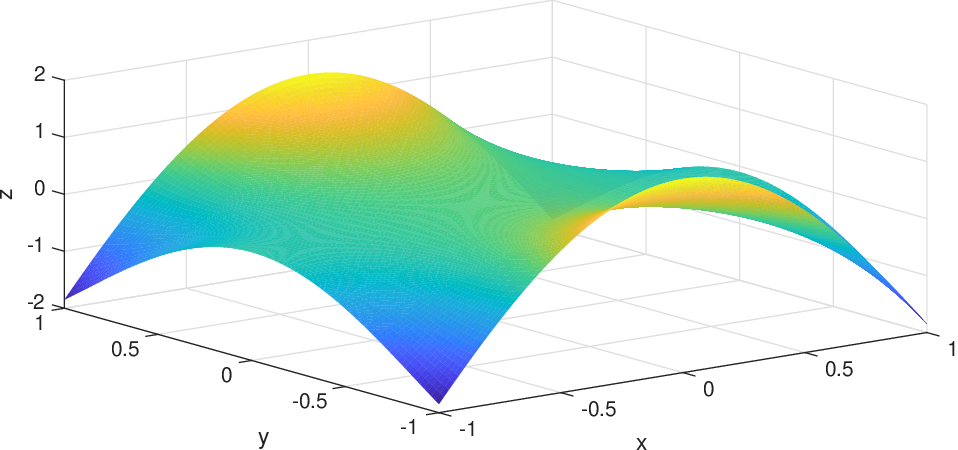}
								\caption{Real part of $f_4$ }
				\end{subfigure}
				\hfill
				\begin{subfigure}[c]{0.32\textwidth}
					\centering
						\includegraphics[width=\linewidth]{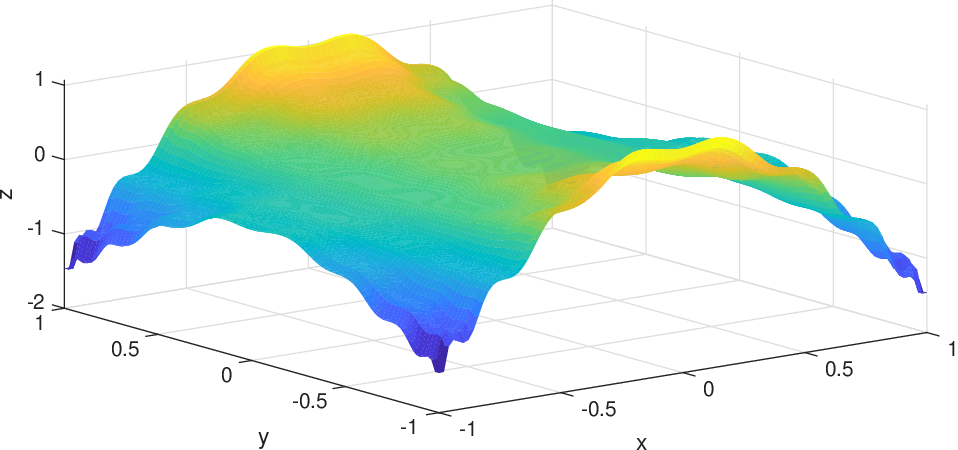}
										\caption{Real part of $f_4$ by $(N_{10,2}f_4)$}
				\end{subfigure}
				\hfill
				\begin{subfigure}[c]{0.32\textwidth}
					\centering
				\includegraphics[width=\linewidth]{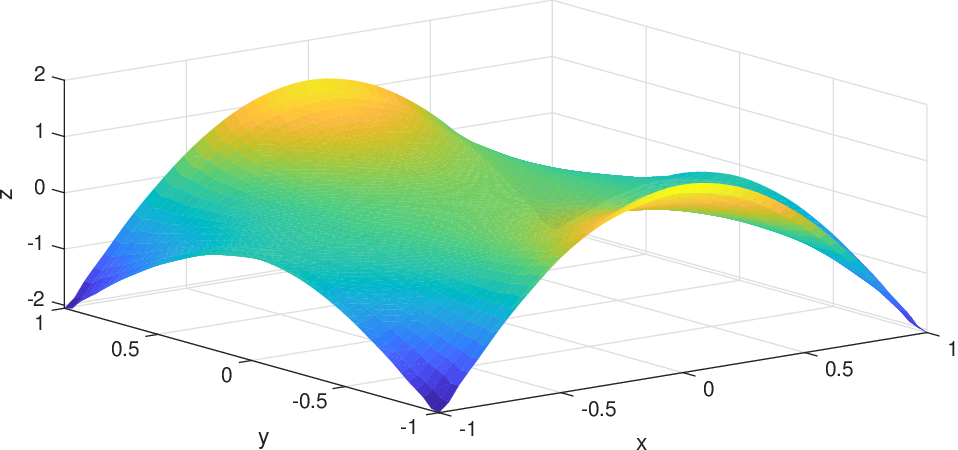}
			\caption{Real part of $f_4$ by $(H^{(3)}_{10,2}f_4)$}
				\end{subfigure}
				
				\vspace{3mm}
				
				\begin{subfigure}[c]{0.32\textwidth}
					\centering
						\includegraphics[width=\linewidth]{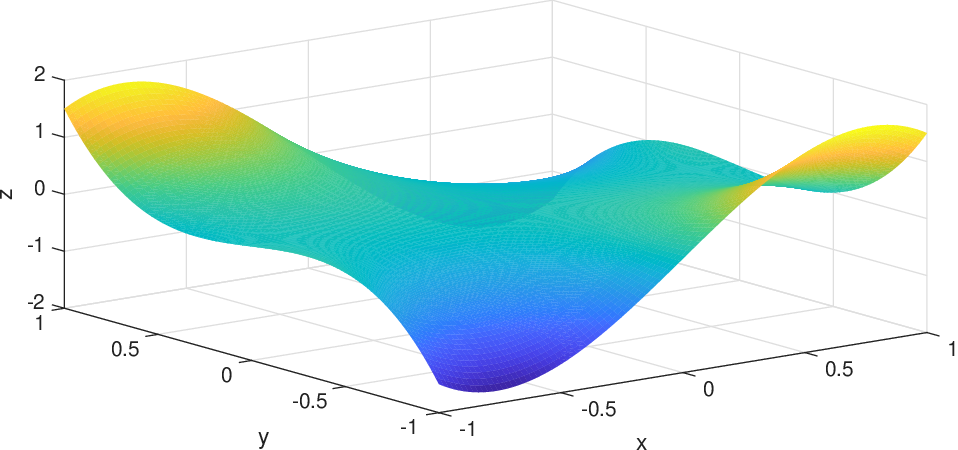}
									\caption{Imaginary part of $f_4$}
				\end{subfigure}
				\hfill
				\begin{subfigure}[c]{0.32\textwidth}
					\centering
						\includegraphics[width=\linewidth]{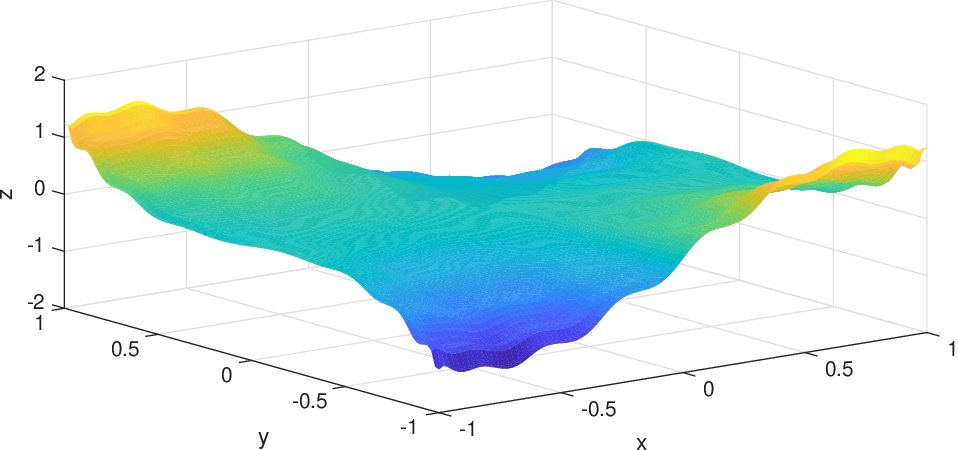}
					\caption{Imaginary part of $f_4$ by $(N_{10,2}f_4)$}
					
				\end{subfigure}
				\hfill
				\begin{subfigure}[c]{0.32\textwidth}
					\centering
					\includegraphics[width=\linewidth]{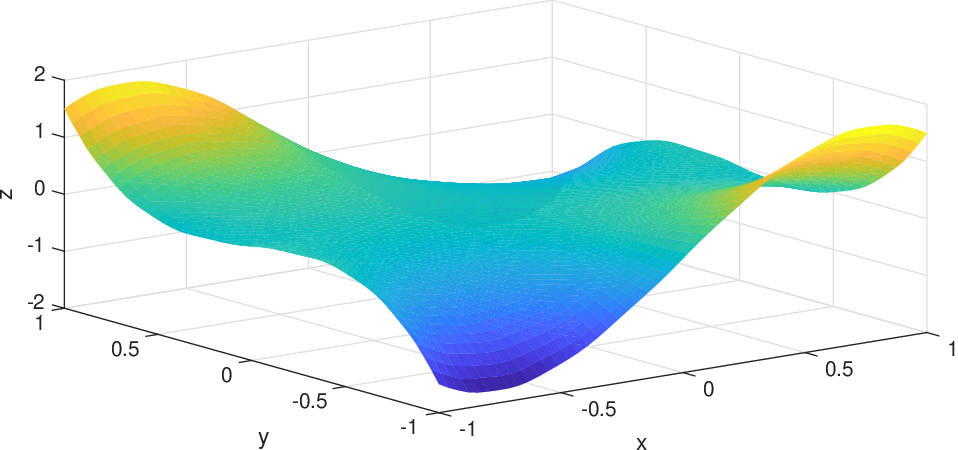}
								\caption{Imaginary part of $f_4$ by $(H^{(3)}_{10,2}f_4)$}
				\end{subfigure}
				
				\caption{Approximation of the real and imaginary  parts of the function $f_4$ by $(N_{n,s}f_4)$ and $(H^{(3)}_{n,s}f_4)$ for $s=2$ and $n=10$  }
							\label{g4}
			\end{figure}

				\begin{table}[ht]
				\centering
				\caption{Errors in approximation of $f_4$ by $(N_{n,2}f_4)$ and $(H^{(3)}_{n,2}f_4)$ for different values of $n$ }
				\label{t4}
				 \renewcommand{\arraystretch}{1.4}
						\begin{tabular}{c| c  c c}
							\hline
							\rowcolor{headerblue}\color{white}
							$n$ & Types of  Operator &  $e_{max}^\Re$  & $e_{max}^\Im$  \\ 
							\hline
							\multirow{2}{*}{$10$} & Generalized Nevai operator (\ref{Nevai}) & 5.180e-01 & 5.323e-01 \\ \cline{2-4} 
							& Hermite type Nevai operator (\ref{herm}) & 3.193e-01  & 2.238e-02\\ \hline
							\multirow{2}{*}{$20$} & Generalized Nevai operator (\ref{Nevai}) & 2.300e-01 & 2.722e-01  \\ \cline{2-4}
							&  Hermite type  Nevai operator (\ref{herm}) & 1.901e-01  & 7.467e-03  \\ \hline
							\multirow{2}{*}{$30$} & Generalized Nevai operator (\ref{Nevai})   & 2.322e-01  & 2.291e-01\\ \cline{2-4}
							&  Hermite type Nevai operator (\ref{herm}) & 1.292e-01  & 3.544e-03  \\ \hline
							\multirow{2}{*}{$40$} & Generalized Nevai operator (\ref{Nevai}) & 1.364e-01  & 1.485e-01 \\ \cline{2-4} 
							&  Hermite type Nevai operator (\ref{herm})  & 9.499e-02 & 2.235e-03  \\
							\hline
						\end{tabular}
					\end{table}
					
					As shown in Figure \ref{g4} and Table \ref{t4}, the complex Hermite type Nevai operators (\ref{herm}) demonstrates superior performance compared to the complex generalized Nevai operators (\ref{Nevai}). These findings indicate that the approximation accuracy might get improved if the higher-order derivatives of the target function are available.
					

			\subsection{Application in Image reconstruction}\label{6.5}
			A complex-valued image is an extension of a standard grayscale image in which each pixel is represented not by a single real number, but by a complex number that encodes both amplitude and phase information. Mathematically, a complex image can be expressed as 
		$f(x,y)=A(x,y) e^{i \phi(x,y)}$ where $A(x,y)$ denotes the amplitude  at pixel $(x,y)$ and $\phi(x,y)$ denotes the phase, representing the angular component of the complex number. The amplitude corresponds to the conventional brightness or magnitude of the pixel, while the phase carries additional structural or wavefront information, which is particularly relevant in applications involving waves, such as optics, holography, and interferometry.  This representation allows one to process and analyze both the magnitude and phase of the image simultaneously, which is essential in many scientific and engineering applications, including wavefront reconstruction, phase imaging, and Fourier-domain signal processing \cite{phase}.

			A complex-valued image with a specific resolution size $u \times v$ is a discrete structure composed of a finite set of pixels captured by an image system, from which a corresponding grayscale image matrix $(c_{ij})_{i,j \in\NN}, i=1,...,u; j=1,...,v$, \hspace{1pt} can be derived. The matrix representation of the gray scale image matrix can be viewed as a two-dimensional step function $A$ in $L^{p}(X)$, where $1 \leq p <\infty.$ The function $A$ is defined as follows
		\begin{equation}\label{ii}
			A(x,y)=\sum_{i=1}^{u}\sum_{j=1}^{v}c_{ij}\cdot \mathbf{1}_{ij}(x,y)
		\end{equation}
		where 	\[
		\mathbf{1}_{ij}(x,y) := 
		\begin{cases} 
			1, &  (x,y)\in (i-1,i]\times (j-1,j], \\
			0, & otherwise.
		\end{cases}
		\]
		Thus, $A(x,y)$ maps every index pair $(i,j)$ to the corresponding value $c_{ij}$.
		
	In this section, we discuss the reconstruction of both the magnitude and phase of an image using the complex Kantorovich type Nevai operators  (\ref{kant}).
			To analyze the quality of the reconstructed image, we utilize the following measures to evaluate the  reconstruction performance:
		\begin{enumerate}
		
	\item 	\textbf{The Structural Similarity Index Measure (SSIM)}: It evaluates the quality of an image primarily by analyzing its brightness, contrast, and structure, defined by
		\[
	SSIM = \frac{(2\mu_H\mu_W + d_1)(2\sigma_{HW} + d_2)}{(\mu_H^2 + \mu_W^2 + d_1)(\sigma_H^2+ \sigma_W^2 + d_2)}.
	\]
		where  $\mu_H$ and $\mu_W$ are the mean values,
		$\sigma_H^2$ and $\sigma_W^2$ are their variances, and
		$\sigma_{HW}$ is the covariance of $H$ and $W$. Moreover,
		$d_1 = (k_1L)^2$, $d_2 = (k_2L)^2$, and
		$L$ is the range of pixel values with $k_1 = 0.01$, $k_2 = 0.03$.\\
%

		
	\item 	\textbf{The Peak Signal to Noise Ratio (PSNR):} The MSE of the reconstructed image
		can be written as
		\[
		MSE = \frac{1}{HW}\sum_{i=1}^{H} \sum_{j=1}^{W}[H(i,j)-W(i,j)]^2,
		\]
		where original and reconstructed images are denoted by \( H \) and \( W \) respectively. Similarly, \( H(i, j) \) and \( W(i, j) \) represent the pixel values at the corresponding coordinates. Moreover, the corresponding PSNR is given by
		$$PSNR = 10 \log_{10} \left( \frac{(MAX)^2}{MSE} \right),$$
		where \( $'$MAX$'$ \) denotes the maximum possible pixel value of the image.\\
		
		\item \textbf{The Root Mean Square Error (RMSE):} It measures the pixel-wise difference between an original image and a reconstructed or processed image.
			For an image of size $H \times W$,  RMSE is defined as
		
		\[
	RMSE = \sqrt{ \frac{1}{HW} \sum_{i=1}^{H} \sum_{j=1}^{W} \Big( H(i,j) - W(i,j) \Big)^2 }.
		\]

		\end{enumerate}

					\begin{figure}[h]
						\centering
						
						\begin{subfigure}{0.22\textwidth}
							\includegraphics[width=\linewidth]{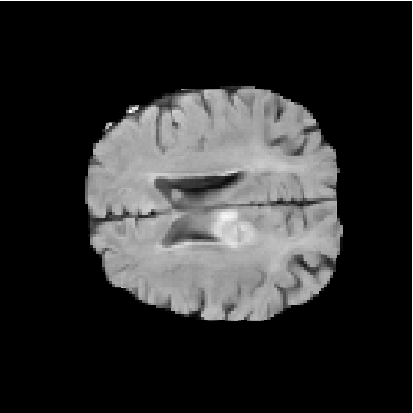}
							\caption{The original brain image\\ (amplitude)}%
						\end{subfigure}
							\begin{subfigure}{0.22\textwidth}
										\includegraphics[width=\linewidth]{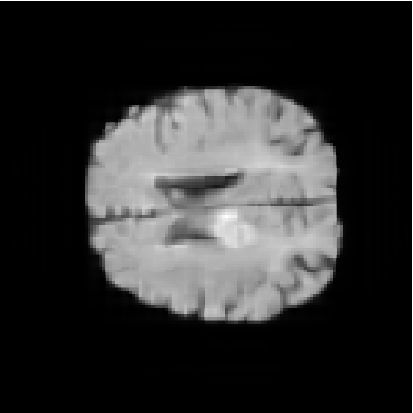}
										\caption{Reconstructed amplitude using (\ref{kant})}
						\end{subfigure}
						\begin{subfigure}{0.22\textwidth}
								\includegraphics[width=\linewidth]{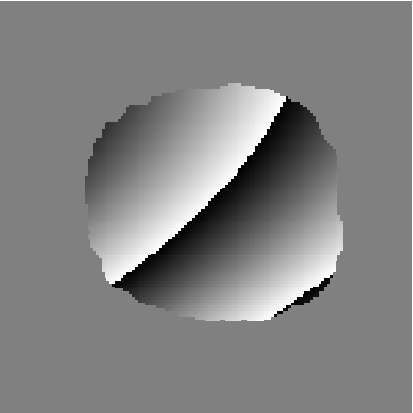}
									\caption{The original brain image\\ (Phase)}
						\end{subfigure}
						\begin{subfigure}{0.22\textwidth}
	\includegraphics[width=\linewidth]{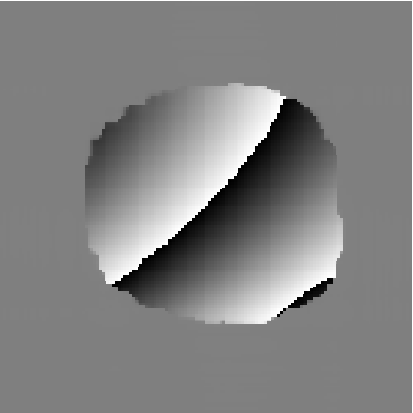}
							\caption{Reconstructed phase using\\ (\ref{kant})}
						\end{subfigure}
						
					
						\caption{Reconstruction of amplitude and phase  using complex Kantorovich type Nevai operators (\ref{kant})  for $n=170$ and $s=2$ }
							\label{fig}
					\end{figure}

					\begin{table}[ht]
						\centering
						\caption{Amplitude and Phase of the brain image using complex Kantorovich type Nevai operators (\ref{kant}) for different $n$ and $s=2$} 
						\label{ap}
						\renewcommand{\arraystretch}{1.4}
						\rowcolors{2}{blue!10}{white}
						\begin{tabular}{>{\bfseries}c c c c c c c}
							\hline
							\rowcolor{headerblue}\color{white}
							$n$ & \multicolumn{3}{c}{Amplitude} & \multicolumn{3}{c}{Phase}  \\ 
							\cmidrule(lr){2-4} \cmidrule(lr){5-7}
							& RMSE & PSNR & SSIM & RMSE & PSNR & SSIM  \\
							\hline
							20 & 0.0902 & 20.09 & 0.4136 & 1.5453 & 20.12 & 0.8175 \\
							60 & 0.0542 & 24.52 & 0.6913 & 0.9958 & 23.94 & 0.8821 \\
							110 & 0.0380 & 27.60 & 0.8423 & 0.7942 & 25.90 & 0.9211\\
							170 & 0.0301 & 29.62 & 0.9199 &  0.3835 & 32.23 & 0.9746 \\
							\hline
						\end{tabular}
					\end{table}

					 Figure~\ref{fig} illustrate the visual quality of the reconstructed image (amplitude and phase) obtained through the complex Kantorovich type Nevai operators (\ref{kant}). 
					Table  \ref{ap} provides the numerical justification of the reconstruction process using standard measures such as PSNR, RMSE and SSIM.

\section{Concluding Remarks}
This paper provides a comprehensive and constructive study on approximation capabilities of a prominent family of complex Nevai operators targeted to approximate analytic as well as non-analytic functions. We construct certain complex interpolation operators such as complex generalized Nevai operators (\ref{Nevai}), complex Kantorovich type Nevai operators (\ref{kant}) and complex Hermite type Nevai operators (\ref{herm}) and examine their convergence behavior for suitable target functions. Some quantitative approximation results are established using appropriate measures such as modulus of continuity and Peetre’s K-functional. Numerical simulations validate the theoretical outcomes as proposed operators approximate various complex-valued functions including analytic functions (see Figure \ref{g1}-\ref{cont1}, Table \ref{t1}-\ref{g11}), integrable functions (see Figure \ref{g2}-\ref{cont2}, Table \ref{t2}-\ref{g22}), and non-analytic functions (see Figure \ref{g3}-\ref{cont3}, Table \ref{t3}-\ref{g33}). Our results indicate that the use of derivative sampling can significantly improve the approximation accuracy, provided that the higher-order derivatives of the target function are available (see Figure \ref{g4}, Table \ref{t4}).
Furthermore, the effectiveness of the complex Kantorovich type Nevai operators (\ref{kant}) in reconstructing real brain image data can be observed in Figure \ref{fig} and Table \ref{ap}.

		\section{Information.} 
	
	\subsection{Ethic.}
	\noindent 
	The authors declare that this is an original work, not previously published or submitted for publication elsewhere. 
	
	\subsection{Funding.} 
	\noindent The research of the first author is funded by the University Grants Commission (UGC), New Delhi, India, through NTA Reference No : 231610034215.
	
	\subsection{Acknowledgement}
	\noindent 
	Priyanka Majethiya and Shivam Bajpeyi gratefully acknowledge SVNIT Surat, India, for the facilities and support provided during the course of this research.

	\subsection{Data availability.} 
	\noindent 
		The data used in this research are obtained from a public medical imaging repository, which is widely used and validated in the scientific community for brain tumor analysis.

	\subsection{Authorship contribution.}
	\noindent
	\textbf{Priyanka Majethiya:} Writing – Original Draft, Writing – Review and Editing, Conceptualization, Formal Analysis, Methodology, Mathematical Proofs, Visualization. 
	\textbf{Shivam Bajpeyi:} Writing – Review and Editing, Conceptualization, Formal Analysis, Methodology, Visualization, Supervision.

	\subsection{Conflict of interest.} 
	\noindent 
	The authors declare that there is no conflict of interest regarding the content of this article.

	\end{document}